\documentclass[11pt]{article}

\usepackage{amsthm}
\usepackage{amsmath}
\usepackage{amsfonts, epsfig, amsmath, amssymb, color, amscd}
\usepackage[cmtip,arrow,all]{xy}
\usepackage{pb-diagram, pb-xy}
\usepackage{amssymb,epsfig}
\usepackage{color}
\usepackage[cmtip,arrow]{xy}
\usepackage{pb-diagram, pb-xy}
\usepackage{amssymb,epsfig,amsfonts}

\usepackage[linktocpage=true,plainpages=false,pdfpagelabels=false]{hyperref}

\usepackage{tikz}
\usepackage{eufrak}

\newfont{\sdbl}{msbm9}
\newfont{\dbl}{msbm10 at 12pt}
\theoremstyle{definition}
\newcommand{\da}{{\mbox{\dbl A}}}

\newcommand{\dz}{{\mbox{\dbl Z}}}

\newcommand{\dn}{{\mbox{\dbl N}}}

\newcommand{\sdn}{{\mbox{\sdbl N}}}

\newcommand{\dc}{{\mbox{\dbl C}}}

\newcommand{\ord}{\mathop{\rm ord}\nolimits}

\newcommand{\End}{\mathop {\rm End}}

\newcommand{\Sup}{\mathop {\rm Supp}}

\newcommand{\rk}{\mathop {\rm rk}}

\newcommand{\Projj}{\mathop {\rm Proj}}

\newcommand{\idm}{\mathfrak{m}}
\newcommand{\Ord}{\mathop {\rm \bf ord}}

\makeatother
\newtheorem{Def}{Definition}[section]
\newtheorem{rem}{Remark}[section]

\newtheorem{ex}{Example}[section]

\theoremstyle{plain}
\newtheorem{Prop}{Proposition}[section]
\newtheorem{theorem}{Theorem}[section]
\newtheorem{lemma}{Lemma}[section]
\newtheorem{cor}{Corollary}[section]

\numberwithin{equation}{section}

\newcommand{\co}{{{\cal O}}}
\newcommand{\crr}{{{\cal R}}}
\newcommand{\cf}{{{\cal F}}}

\newcommand{\cm}{{{\cal M}}}

\begin{document} 
\title{Normal forms for ordinary differential operators, I}
\author{J. Guo \and  A.B. Zheglov}
\date{}
\maketitle


\begin{abstract}
In this paper we develop the generalised Schur theory offered in the recent paper by the second author in  dimension one case, and apply it to obtain  a new explicit parametrisation of torsion free rank one sheaves on projective irreducible curves with vanishing cohomology groups.

This parametrisation is obtained with the help of normal forms - a notion we introduce in this paper. Namely, considering the ring of ordinary differential operators $D_1=K[[x]][\partial ]$ as a subring of a certain complete non-commutative ring $\hat{D}_1^{sym}$, the normal forms of differential operators mentioned here are obtained after conjugation by some invertible operator ("Schur operator"), calculated using one of the operators in a ring. Normal forms of commuting operators are polynomials with constant coefficients in the differentiation, integration and shift operators, which have a restricted finite order in each variable, and can be effectively calculated for any given commuting operators. 

\end{abstract}


\markright{Normal forms for ODOs}

\tableofcontents

\section{Introduction}
\label{S:introduction}

In this paper we develop the generalised Schur theory offered in \cite{A.Z} in  dimension one case and apply it to obtain  a new convenient explicit parametrisation of torsion free rank one sheaves on projective irreducible curves with vanishing cohomology groups. 

It is well known  that such sheaves are exactly the spectral sheaves of commutative subrings of ordinary differential operators of rank one (see e.g. review \cite{Zheglov_book} for detailed exposition of the theory of commuting ordinary differential operators or \cite{BZ} for a short explanation). This result is the first step in establishing similar parametrisation  for spectral sheaves of arbitrary rank, and also for spectral sheaves of commutative subrings of operators in higher dimensions (cf. \cite{BurbanZheglov2017}, \cite{Zheglov_belovezha}). It is motivated by an important problem that appears in algebraic-geometric classification  of commutative subrings of operators in higher dimensions -- a description of torsion free sheaves with specific conditions on cohomology groups (see \cite{Zheglov_belovezha}), in particular with fixed Hilbert polynomial and some extra conditions. 
In the work \cite{BurbanZheglov2017} a description of Cohen-Macaulay sheaves on the spectral surface of quantum Calogero-Moser systems was given  with the help of matrix problem approach due to I. Burban and Y. Drozd \cite{SurvOnCM}, \cite{BDNonIsol} (Cohen-Macaulay sheaves form an open part of the moduli space of torsion free sheaves with fixed Hilbert polynomial, and by that reason it is important  to describe them at first), however this approach meets with a difficulty to describe sheaves with specific cohomological properties (cf. \cite[Sec. 6]{BurbanZheglov2017}). We expect that our new approach will help to solve this problem in an effective way in any dimension.

The necessity of further development of the Schur theory was not restricted only by the above mentioned application. Let's recall two major theorems of this theory from \cite{A.Z} (we formulate them here in a simplified form, for notations see the List of notations below):
\begin{theorem}[A generalized Schur theorem, Th. 7.1 ]
\label{T:Schur_theorem}
Let $P_1, \ldots , P_n\in \hat{D}_n^{sym}$ be commuting operators with $\Ord (P_i)=k$ for all $i=1,\ldots ,n$.  Assume that the module $F$ of the ring $K[\sigma (P_1), \ldots , \sigma (P_n)]$ is finitely generated and free. 

Then there exists an invertible operator $S\in \hat{D}_n^{sym}$ with $\Ord (S)=0$ such that
$$
S^{-1}\partial_i^{k}S=P_i, \quad i=1, \ldots ,n. 
$$
\end{theorem}

If $n=1$ then the conditions of the generalized Schur theorem are automatically satisfied for any {\it monic} operator $P\in \hat{D}_1^{sym}$. So, $P=S\partial^q S^{-1}$ for some $S\in \hat{D}_1^{sym}$.

\begin{theorem}[A centralizer theorem, Prop.7.1 ]
\label{T:centraliser_theorem}
Assume $[\partial_q^{k}, Q]=0$ for $q=1, \ldots , n$, where $Q\in \hat{D}_n^{sym}$. Then 
$$
Q=\sum_{j_{1}=0}^{k-1}\ldots \sum_{j_{n}=0}^{k-1}c_{j_1,\ldots ,j_n}A_{k; j_1,1}\ldots A_{k; j_n,n},
$$
where all coefficients  $c_{j_1,\ldots ,j_n}$ are {\it polynomials} in $\partial_q$, $\int_q$, $q=1,\ldots , n$ with {\it constant} coefficients  and the degree of these polynomials with respect to $\partial_q$ is not greater than $\Ord (Q)$ and the degree of these polynomials with respect to $\int_q$ is not greater than $k-1$.
\end{theorem}

If the operators $P_i$ from the first theorem are {\it differential}, i.e. $P_i\in D_n$, a natural question appears: what is the shape of the operator $S$? In  dimension one case it is well known (see e.g. \cite{Mu}, \cite{Li_Mulase}, \cite{SW} or the book  \cite{Zheglov_book} and references therein) that, if the centraliser of $P_1$ is non-trivial, then  $S$ can be expressed via the Baker-Akhieser function (and vice versa), though in quite non-trivial way, if the rank of the centraliser is greater than one. Starting from works \cite{Kr1}, \cite{Kr2}, \cite{KN} it is known that  the Baker-Akhieser function plays an important role in many problems of mathematical physics, in particular it played a key role in constructing explicit examples of commuting operators in many works (cf. \cite{Mokh}, \cite{M1} for examples in dimension one, and \cite{ChalykhVeselov90}, \cite{ChalykhVeselov98}, \cite{ChalykhFeiginVeselov2} for examples in higher dimensions). Analogously, in higher dimension the  Schur-Sato operator $S$ determines the common eigenfunction of commuting operators (of different nature), see \cite[Sec. 6]{BurbanZheglov2017}, \cite{Zheglov_belovezha}, cf. also \cite{BerestChalykh22} and references therein, and the knowledge of its shape could help to prove the classification conjecture \cite[Conj. 7.11]{Zheglov_belovezha} about commuting partial differential operators. Besides, in higher dimension, the operator $S$ determines an order-preserving endomorphism of the Weyl algebra (\cite[Cor. 2.1]{A.Z}), thus giving hint to the Dixmier and Jacobian conjectures\footnote{ Occasionally, the technique of normal forms developed in this paper appeared to be very useful in proving the Dixmier conjecture for the first Weyl algebra offered in a recent preprint by the second author: A.B. Zheglov, {\it The conjecture of Dixmier for the first Weyl algebra is true}, https://arxiv.org/abs/2410.06959 . We hope that further development of the normal form technique together with effective methods of positive characteristic from papers \cite{BK}, \cite{BK1}, \cite{Ts1}, \cite{Ts2}, used in proving the equivalence of the Dixmier conjectures and the Jacobian conjectures, will allow people to come closer to solving at least the $JC_2$ conjecture. }. 

Analysing the shape of the Schur operator $S$, we find that all its homogeneous components are non-commutative polynomials with constant coefficients in the differentiation, integration, shift operators $A_i$ (see below) and the operator $\Gamma :=x\partial$ (we call such polynomials as HCPC for short), which have a finite order in each variable and additionally satisfy a specific property of being totally free of $B_j$ (see definition \ref{D:total_B}). All HCPC  form a subring $Hcpc(k)$, which occasionally looks very similar to the algebra of polynomial integro-differential operators studied in the paper \cite{Bavula} (though the shift operators are not included into this algebra). We establish estimates on the degrees of these polynomials in section \ref{S:necessary conditions on the Schur}. As a result we encode all necessary  properties of the operator $S$ in a condition $A_q(k)$ of section \ref{S:normal forms} (see  definition \ref{D:conditionA} or the list of notations below). With the help of this condition, we gave a criterion of an operator $P\in \hat{D}_1^{sym}$ to be a differential operator (see theorem \ref{T: P in hat(D) is a dif_op}). 

{\it A normal form} of a pair of operators $P,Q\in \hat{D}_1^{sym}$ is a pair $P', Q' \in  \hat{D}_1^{sym}$ obtained after conjugation by a 
Schur operator $S$ as above, calculated using one of the operators in a pair $(P,Q)$ (or, more generally, in the ring $\hat{D}_1^{sym}$, see definitions \ref{D:normal_form}, \ref{D:normalised_normal_form} and remark \ref{R:normal_forms}). The normal form is not uniquely defined, but up to conjugation with invertible $S\in \hat{D}_1^{sym}$ from the centralizer $C(\partial^k)$  with $\Ord (S)=0$. By centralizer theorem \ref{T:centraliser_theorem}  such $S$ is a {\it polynomial} of restricted degree. Notably, the whole centraliser $C(\partial^k)$ is naturally isomorphic to a matrix $k\times k$ algebra over a polynomial ring, see remark \ref{R:algebraic_dependence_of_normal_forms}.
The normal form of {\it commuting operators} can be normalized in some way (see section \ref{S:normal_forms_for_commuting}). By the centralizer theorem $C(\partial^k)$ consists of (non-commutative) {\it polynomials}, so  normalised normal form can be calculated for any such operators. If the operators do not commute, the normal forms will  be series in general, for which, however, it is possible to calculate any given number of terms. For a pair of {\it differential operators} normal forms satisfy condition $A_q(0)$ (see corollary \ref{C:P'}).

Normal forms of a commutative subring $B\subset D_1$ appear to be a very effective tool to describe the {\it moduli space of spectral sheaves}, i.e. torsion free sheaves on the spectral curve with certain conditions on cohomology groups, cf. \cite[\S 1.3]{BZ} and theorem \ref{T:classif2} below. Roughly speaking, the set of coefficients of a normalised normal form determines such a sheaf up to an isomorphism. This set depends on a choice of normalisation, and can be thought of as a system of local coordinates on a chart of a manifold -- the moduli space of spectral sheaves. Precise statement about this parametrisation in case of sheaves of rank one is formulated in theorem  \ref{T:parametrisation}.

We will consider  a similar description of the moduli space for sheaves of rank $>1$ in the next article, since this case requires much more details. 
We expect that further study of normal forms is reasonable not only for differential operators, but also for operators of other type, like difference, integro-differential, etc. In particular, it seems to be promising  to study normal forms of recently discovered examples from \cite{MironovMaul}, \cite{MironovMaul1}.

\bigskip

The structure of this article is the following. In section \ref{S:generic_normal_forms} we develop generic theory of normal forms for ordinary differential operators. Namely, in section \ref{S:Schur theory for the ring}  we review the Schur theory from \cite{A.Z} in the case of dimension one, strengthening some specific statements useful in the rest of the paper. In section \ref{S:Basic  formulae in} we deduce a list of useful formulae, in section \ref{S:HCP} we introduce an important technical notion of homogeneous canonical polynomials (HCP) and study their basic properties. This section contains important estimations and formulae useful for fast explicit calculations of concrete examples of normal forms and Schur operators. In section \ref{S:necessary conditions on the Schur}  we develop the Schur theory further by studying necessary conditions on the Schur conjugating operators for ordinary differential operators. In section \ref{S:normal forms} we introduce the main subject of this section -- normal forms for differential operators and study basic properties of them. 
The major result of this section is theorem \ref{T: P in hat(D) is a dif_op}, a criterion of an operator from $\hat{D}_1^{sym}$ to be a differential operator. 

In section \ref{S:normal_forms_for_commuting}  we study normal forms of commuting differential operators with the help of theory from section \ref{S:generic_normal_forms}. We give a convenient description of the centraliser and of normal forms of a pair of commuting operators. 

With the help of this description  we prove our main result, theorem \ref{T:parametrisation}, where a new parametrisation of torsion free sheaves of rank one with vanishing cohomology groups on a projective  curve is given.  In particular, this result indicates that the moduli space of {\it spectral sheaves} of rank one is an affine open subscheme of the compactified Jacobian (cf. remark \ref{R:moduli}).

 The paper (especially section 3) makes substantial use of the theory of commuting ODOs, which has a rich history, and its generalisations offered in a series of papers by the second author and other collaborators. To better understand the ideas behind this work, we recommend reading the background of the recent review \cite{Zheglov_book} (especially sections 9-11) and articles \cite{BurbanZheglov2017}, \cite{A.Z}.

{\bf Acknowledgements.}  The work of J. Guo was partially supported by the National Key R and D Program of China (Grant No. 2020YFE0204200). The work of A. Zheglov was partially supported by RSF grant no. 25-11-00210. 

A. Zheglov is grateful to the Sino-Russian Mathematics Center at Peking University for hospitality and excellent working conditions while preparing this paper.

We are grateful to Huijun Fan for his interest in our work, and for his advice and support to J. Guo. 

We are also grateful to the anonymous referees whose remarks allowed to improve the exposition.

\subsection{List of notations}

 Since this work uses quite different techniques, for convenience of the reader we introduce now the most important notations used in this paper.
 
1. $\dz_+$ is the set of all non-negative integers, $\dn$ is the set of natural numbers (all positive integers). $K$ is a field of characteristic zero. Recall some notation from \cite{A.Z}:
	$\hat{R}:=K [[x_1,\ldots ,x_n]]$, the $K$-vector space 
$$
\cm_n := \hat{R} [[\partial_1, \dots, \partial_n]] = \left\{
\sum\limits_{\underline{k} \ge \underline{0}} a_{\underline{k}} \underline{\partial}^{\underline{k}} \; \left|\;  a_{\underline{k}} \in \hat{R} \right. \;\mbox{for all}\;  \underline{k} \in \dn_0^n
\right\},
$$
$\upsilon:\hat{R}\rightarrow \dn_0\cup \infty$ --  the discrete valuation defined by the unique maximal ideal $\idm = (x_1, \dots, x_n)$ of $\hat{R}$,  \\
for any element
$
0\neq P := \sum\limits_{\underline{k} \ge \underline{0}} a_{\underline{k}} \underline{\partial}^{\underline{k}} \in \cm_n
$
$$
\Ord (P) := \sup\bigl\{|\underline{k}| - \upsilon(a_{\underline{k}}) \; \big|\; \underline{k} \in \dn_0^n \bigr\} \in \dz \cup \{\infty \},
$$
$$
\hat{D}_n^{sym}:=\bigl\{Q \in \cm_n \,\big|\, \Ord (Q) < \infty \bigr\};
$$
$
P_m:= \sum\limits_{ |\underline{i}| - |\underline{k}| = m} \alpha_{\underline{k}, \underline{i}} \,  \underline{x}^{\underline{i}} \underline{\partial}^{\underline{k}}
$ -- the $m$-th \emph{homogeneous component} of $P$,\\
$\sigma (P):=P_{\Ord (P)}$ -- the highest symbol.
	
2. In this paper we use: 	
	$\hat{R}:=K[[x]]$, $D_1:=\hat{R}[\partial]$, 
$$\hat{D}_1^{sym}:=\{Q=\sum_{k\ge 0}a_k\partial^k|\Ord(Q)<\infty\}.$$ 

Operators:  $\delta:=\exp((-x)\ast \partial)$\footnote{ Here and further $\ast$ in all exponentials means that we consider normalized Taylor power series, i.e. the powers of $x$ always stand on the left of powers of $\partial$, for example $\delta:=\exp((-x)\ast \partial)=\sum_{k=0}^{+\infty}\frac{(-1)^k}{k!}x^k\partial^k$.},   $\int:=(1-exp((-x)\ast \partial))\partial^{-1} $,  	$A_{k;i}:=\exp((\xi^{i}-1)x\ast \partial)\in \hat{D}^{sym}_{1}\hat{\otimes}_{K}\tilde{K}$ (in the case when $k$ is fixed, simply written as $A_i$), where $\tilde{K}=K[\xi]$, $\xi$ is a primitive $k$th root of unity, 
$\Gamma_i=(x\partial)^i$. $B_n=\frac{1}{(n-1)!}x^{n-1}\delta\partial^{n-1}$. 

$\hat{D}^{sym}_{1}\hat{\otimes}_{K}\tilde{K}$ means the same ring $\hat{D}^{sym}_{1}$, but defined over the base field $\tilde{K}$.

The operator $P\in D_1$ is called {\it normalized} if $P=\partial^p+a_{p-2}\partial^{p-2}+\ldots $. The operator $P\in D_1$ is {\it monic} if its highest coefficient is 1. Analogously, $P\in \hat{D}_1^{sym}$ is monic if $\sigma (P)=\partial^p$. 

	 We denote $D^i=\partial^i$ if $i\ge 0$ and $\int^{-i}$ if $i<0$. Operators written in the (Standard) form as 
	$$
	H=[\sum_{0\leq i<k}f_{i;r}(x, A_{k;i}, \partial )+\sum_{0<j\leq N}g_{j;r}B_{j}]D^{r}
	$$
	are called HCP and  form a sub-ring $Hcpc(k)$.   Here $f_{i;r}(x, A_{k;i}, \partial)$ is a polynomial of $x, A_{k;i},\partial$,  $\Ord(f_{i;r})=0$,  of the form
		$$
		f_{i;r}(x, A_{k;i}, \partial )=\sum_{0\leq l\leq d_{i}}f_{l,i;r}x^l A_{k;i}\partial^l
		$$
		for some $d_i\in \dz_+$, where $f_{l,i;r}\in \tilde{K}$. The number  $d_i$ is called the {\it $x$-degree of $f_{i;r}$}:  $deg_{x}(f_{i;r}):= d_i$; $g_{j;r}\in \tilde{K}$, $g_{j;r}=0$ for $j\le -r$ if $r<0$.

	They can be written also in G-form: 
	$$
	H=(\sum_{0\leq i<k}\sum_{0\leq l\leq d_i} f'_{l,i;r}\Gamma_lA_i+\sum_{0<j\leq N}g_{j;r}B_{j})D^{r}
	$$
	The $A$ and $B$ Stable degrees of HCP are defined as 
	$$
	Sdeg_A(H)=\max \{d_i|\quad 0\leq i<k \} \quad \mbox{or $-\infty$, if all $f_{l,i;r}=0$ } 
	$$
	and 
	$$Sdeg_B(H)=\max\{j|\quad g_{j;r}\neq0\} \quad \mbox{or $-\infty$, if all $g_{j;r}=0$}
	$$
	
	In the case when $Sdeg_B(HD^p)=-\infty,\forall p\in \mathbf{Z}$ $H$ is called {\it totally free of} $B_j$.
	
	An operator $P\in  \hat{D}_1^{sym}$ satisfies {\it condition $A_q(k)$}, $q,k\in \dz_+$, $q>1$ if 
	\begin{enumerate}
		\item
		$P_{t}$ is a HCP  from $Hcpc (q)$  for all $t$;
		\item
		$P_{t}$ is totally free of $B_j$ for all $t$;
		\item
		$Sdeg_A(P_{\Ord (P)-i})< i+k$ for all $i>0$;
		\item
		$\sigma (P)$ does not contain $A_{q;i}$, $Sdeg_A(\sigma (P))=k$.
	\end{enumerate}

3. In section 3, $\mathfrak{B}=\crr_S$ is the right quotient ring of $\crr = \tilde{K}^{\oplus k} [D,\sigma ]$ by $S=\{D^k|k\ge 0\}$. And the ring of skew pseudo-differential operators 
	$$
	E_k:=\tilde{K}[\Gamma_1, A_1]((\tilde{D}^{-1}))=\{\sum_{l=M}^{\infty}P_l\tilde{D}^{-l} | \quad P_l\in \tilde{K}[\Gamma_1, A_1]\} \simeq \tilde{K}^{\oplus k}[\Gamma_1]((\tilde{D}^{-1}))
	$$
with the commutation relations 
$$
\tilde{D}^{-1}a=\sigma (a)\tilde{D}^{-1}, \quad a\in \tilde{K}[\Gamma_1, A_1] \quad \mbox{where \quad }
\sigma (A_1)= \xi^{-1} A_1, \quad \sigma (\Gamma_1)=\Gamma_1+1.
$$	
	
	$\widehat{Hcpc}_B(k)$ is the $\tilde{K}$-subalgebra in $\hat{D}_1^{sym}\hat{\otimes}\tilde{K}$ consisting of operators whose homogeneous components are HCPs totally free of $B_j$. 
	
	$$\Phi : \tilde{K}[A_1,\ldots ,A_{k-1}]\rightarrow \tilde{K}^{\oplus k} , \quad  { P\mapsto (\sum_i p_i\xi^{i}, \ldots , \sum_i p_i\xi^{i(k-1)})}
$$
is an isomorphism of $\tilde{K}$-algebras. 	

The map 
$$
\hat{\Phi}: \widehat{Hcpc}_B(k) \hookrightarrow E_k
$$
denotes an embedding of $\tilde{K}$-algebras (see lemma \ref{L:embedding_Phi}). 
	
	Suppose $B$ is a commutative sub-algebra of $D_1$, then $(C,p,\cf)$ stands for the spectral data of $B$ (the spectral curve, point at infinity and the spectral sheaf with vanishing cohomologies).
	
	The classical ring of pseudo-differential operators  is defined as
	$$
E=K[[x]]((\partial^{-1})).
	$$

There is an isomorphism of $\tilde{K}$-algebras $\psi:\mathfrak{B}\rightarrow M_k(C(\mathfrak{B}))$, where $C(\mathfrak{B})\simeq \tilde{K}[\tilde{D}^k, \tilde{D}^{-k}]$, ($\tilde{K}$ is diagonally embedded into $\tilde{K}^{\oplus k}$):
	$$
	\psi\begin{pmatrix}
		h_0 \\
		h_1 \\
		\cdots \\
		h_{k-1}
	\end{pmatrix}=\begin{pmatrix}
		h_0 &  &  &  \\
		& h_1 &  &  \\
		&  & \cdots &  \\
		&  &  & h_{k-1}
	\end{pmatrix}\quad \psi(D)=T:=\begin{pmatrix}
		& 1 &  & \cdots &  \\
		&  & 1 & \cdots &  \\
		\cdots & \cdots & \cdots & \cdots & \cdots \\
		&  &  & \cdots & 1 \\
		D^k &  &  & \cdots & 
	\end{pmatrix}
	$$
	with $\psi(D^l)=T^l$, and extended by linearity. The map $\psi$ can be obviously extended to 
	$$
\psi : \tilde{K}^{\oplus k}((\tilde{D}^{-1})) \hookrightarrow M_k(\tilde{K}((\tilde{D}^{-k}))).
	$$

\section{Generic theory of normal forms}
\label{S:generic_normal_forms}

\subsection{Preliminary statements from the Schur theory for the ring $\hat{D}_1^{sym}$}
\label{S:Schur theory for the ring}

Suppose $K$ is a field of characteristic zero. 

Following the notations in \cite{A.Z}, denote $\hat{R}:=K[[x]]$, $D_1:=\hat{R}[\partial ]$, define the $K$-vector space
$$
\mathcal{M}_{1}:=\hat{R}[[\partial]]=\{\sum_{k\ge 0}a_{k}\partial^{k}|a_{k}\in \hat{R}\quad  \forall k \in \mathbb{N}_{0}\} ,
$$
where $v:\hat{R}\rightarrow \mathbb{N}_0 \cup \infty$ is the discrete valuation defined by the unique maximal ideal of $\hat{R}$, for any element $0\neq P:=\sum_{k \ge 0}a_{k}\partial^{k}\in \mathcal{M}_{1}$ define the order function 
$$
\Ord(P):=\sup\{k-v(a_{k})|\quad k\in \mathbb{N}_{0}\}\in \mathbb{Z}\cup \{\infty\}.
$$
Define the space
$$
\hat{D}^{sym}_{1}:=\{Q\in\mathcal{M}_1|\quad \Ord(Q)<\infty \} .
$$
By definition, any element $P\in \hat{D}^{sym}_{1}$ is written in the {\it canonical form} 
$$
P:=\sum_{k-i\le \Ord (P)}\alpha_{k,i}x^{i}\partial^{k}.
$$ 
We call $P_{m}:=\sum_{k-i=m}\alpha_{k,i}x^{i}\partial^{k}$ as the $m$-th {\it homogeneous component} of $P$, we call $\sigma(P):=P_{\Ord(P)}$ as the {\it highest symbol} of $P$. Then we have the (uniquely defined) {\it homogeneous decomposition} for any $P\in \hat{D}^{sym}_{1}$:
$$
P=\sum_{m=-\infty }^{\Ord(P) }P_{m}. 
$$
Denote by $A_1:=K[x][\partial ]$ the first Weyl algebra. Clearly, $A_1\subset D_1 \subset \hat{D}_1^{sym}$. 

\begin{rem}
\label{R:ord=deg}
	Note that the order function $\Ord$ coincide with the weight function $v_{1,-1}$ on the ring $A_1$ from the paper \cite{Dixmier}.  For any operator $P \in D_1$, we define the usual order (or degree) of  $P$  as $\deg (P):=v_{0,1}(P)$. If $P\in D_1$ has an invertible highest coefficient (with respect to the usual order), then it is easy to see that  $\deg (P)=\Ord (P)$. Vice versa, any $P\in D_1$  with $\deg (P)=\Ord (P)$ has an invertible highest coefficient.
\end{rem}

\begin{theorem}{(\cite{A.Z},Theorem 2.1)}
	\label{T:A.Z 2.1}
	The following statement are properties of $\hat{D}^{sym}_{1}$
	\begin{enumerate}
		\item $\hat{D}^{sym}_{1}$ is a ring (with natural operations $\cdot$ ,+ coming from $D_{1}$); $\hat{D}^{sym}_{1} \supset D_{1}$.
		\item $\hat{R}$ has a natural structure of a left $\hat{D}^{sym}_{1}$-module, which extends its natural structure of a left $D_{1}$-module.
		\item We have a natural isomorphism of $K$-vector spaces 
		$$
		F:=\hat{D}^{sym}_{1}/ \mathfrak{m}\hat{D}^{sym}_{1}\rightarrow K[\partial]
		$$
		where $\mathfrak{m}=(x)$ is maximal ideal of $\hat{R}$.
		\item Operators from $\hat{D}^{sym}_{1}$ can realise arbitrary endomorphisms of the $K$-algebra $\hat{R}$ which are continuous in the $\mathfrak{m}$-adic topology\footnote{ We consider here the representation of the ring  $\hat{D}_1^{sym}$ given in item 2, in terms of this representation any continuous endomorphism can be represented by an operator from $\hat{D}_1^{sym}$ (the details can be found in \cite{A.Z}), so $\hat{D}_1^{sym}\supset \End_{K-alg}^c(\hat{R})$, however, say, the operator $\partial$ is not a $K$-algebra endomorphism}.
		\item There are Dirac delta functions, operators of integration, difference opertors.
	\end{enumerate}
\end{theorem}

The Dirac delta is given by the series $\delta:=exp((-x)\ast \partial):= 1- x\partial +\frac{1}{2!}x^2\partial^2-\ldots$,  which satisfies $\delta \circ f(x)=f(0)$\footnote{ We denote by $\circ$ here and further  the action of operators from $\hat{D}_1^{sym}$ on elements of $\hat{R}$, determined by the module structure from item 2 of \ref{T:A.Z 2.1}. We denote by $\cdot$ the multiplication in the ring $\hat{D}_1^{sym}$. Sometimes the sign $\cdot$ is omitted, and in these cases we  always mean $\cdot$, not $\circ$. 
}, and the operator of integration is given by the series
$$\int:=(1-exp((-x)\ast \partial))\cdot \partial^{-1}=\sum_{k=0}^{\infty}\frac{x^{k+1}}{(k+1)!}(-\partial)^{k} $$
which satisfies  
$$
\int \circ x^{m}=\frac{x^{m+1}}{m+1}.
$$
Note that  $\int$ is only the right inverse of $\partial$, because $\partial \int =1$ and  $\int  \partial =((1-exp((-x)\ast \partial)) \partial^{-1} \partial)=1-\delta$. 

\begin{rem}
\label{R:ord_properties}
Unlike the usual ring of PDOs the ring $\hat{D}_1^{sym}$ contains zero divisors. There are the following obvious properties of the order function (cf. the proof of \cite[Th. 5.3]{BurbanZheglov2017}):
\begin{enumerate}
\item
$\Ord (P Q)\le \Ord (P)+\Ord (Q)$, and the equality holds if $\sigma (P)\sigma (Q)\neq 0$,
\item
$\sigma (P Q)=\sigma (P) \sigma (Q)$, provided $\sigma (P) \sigma (Q)\neq 0$,
\item
$\Ord (P+Q)\le \max\{\Ord (P),\Ord (Q)\}$.
\end{enumerate}
In particular, the function $-\Ord$ determines  a discrete pseudo-valuation on the ring $\hat{D}_1^{sym}$. 
\end{rem}

\begin{Def}
\label{D:Aki}
	Let $\xi$ be a primitive $k$-th root of unity, $\tilde{K}=K[\xi]$. For any $i\in \mathbb{Z}$, we define operators 
	$$
	A_{k;i}:=\exp((\xi^{i}-1)x\ast \partial)\in \hat{D}^{sym}_{1}\hat{\otimes}_{K}\tilde{K},
	$$
	where $\hat{D}^{sym}_{1}\hat{\otimes}_{K}\tilde{K}$ means the same ring $\hat{D}^{sym}_{1}$, but defined over the base field $\tilde{K}$. 
	
	Further, if it will be clear from the context, we'll omit index $k$ and use notation $A_i:=A_{k;i}$.
\end{Def}

\begin{lemma}{(cf. \cite{A.Z},Lemma 7.2)}
	\label{T:A.Z L 7.2}
	The sum $$
	A:=c_0+c_1 A_{k,1}+\dots+c_{k-1}A_{k;k-1},\quad c_{i}\in \tilde{K}
	$$
	is equal to zero iff $c_{i}=0, i=0,\dots,k-1$. If it is not equal to zero, then it is of order zero. 
	
	Moreover, $A$ is a polynomial in $\partial$ iff $c_1=\ldots = c_{k-1}=0$.
\end{lemma}

\begin{proof}
The first part of lemma coincides with \cite[Lem. 7.2]{A.Z}. The last assertion follows easily from the proof of \cite[Lem. 7.2]{A.Z}. Namely, $A$ is a polynomial iff the infinite system of linear equations hold 
$$
c_1(\xi-1)^j+\ldots +c_{k-1}(\xi^{k-1}-1)^j=0, \quad j\ge n_0\in \dn
$$
for an appropriate $n_0$. But by the well known property of the Vandermonde matrix this system has the unique solution $c_1=\ldots =c_{k-1}=0$. 

\end{proof}

The following claim is a partial case of \cite[Prop. 7.1]{A.Z}.

\begin{Prop}
	\label{T:A.Z 7.1}
	In the notation of definition \ref{D:Aki} we have 
	\begin{enumerate}
		\item For any $i,j,p$, we have
	$$A_{i} A_{j}=A_{i+j}, \quad \partial^{p}A_{i}=\xi^{pi}A_{i}\partial^{p}, \quad A_{i}x^{p}=\xi^{pi}x^{p}A_{i},\quad \int^{p}A_{i}=\xi^{-pi}A_{i}\int^{p}.$$
		\item For a given $Q \in  \hat{D}^{sym}_{1} $ assume that $[\partial^{k},Q]=0 $. 
		\\Then $Q=c_{0}+c_{1}A_{1}+\dots+c_{k-1}A_{k-1}\in  \hat{D}^{sym}_{1}\hat{\otimes}_{K}\tilde{K}$, where $c_{i}\in  \hat{D}^{sym}_{1}\hat{\otimes}_{K}\tilde{K}$ are given by the following formula:
		$$c_{i}=\sum_{m=0}^{\Ord (Q)}c_{i,m}\partial^{m}+c_{i,-1}\int+\dots+c_{i,-k+1}\int^{k-1} $$ 
		where $c_{i,j}\in \tilde{K}$ (so that $\Ord(c_{i})\le \Ord(Q)$). Besides, the coefficients $c_{i,j}$ are uniquely defined.
	\end{enumerate}
	
\end{Prop}

\begin{proof}
For convenience of the reader we'll give here a proof of item 2 which is easier in our case than the  proof of  \cite[Prop. 7.1]{A.Z} (and we'll use some of its arguments later).

The identity $[\partial^k, Q]=0$ can be rewritten as 
\begin{equation}
\label{E:comm}
\sum_{i=1}^k \binom{k}{i}  \partial^i(Q)\partial^{k-i}=0. 
\end{equation}
Note that any solution $Q\in \hat{D}_1^{sym}$ of this equation gives a solution 
$Q'\in \tilde{K}[[x, \tilde{\partial}]]$ of the equation  
$$
\sum_{i=1}^k \binom{k}{i} \partial^i(Q')\tilde{\partial}^{k-i}=0,
$$
where $\tilde{\partial}$ means a new formal variable (commutative with  $x$). Namely, we just replace $\partial$ by $\tilde{\partial}$ in the series $Q$. 

On the other hand, the last equation can be written in the form 
$$
\prod_{i=1}^k(\partial+ (1-\xi_k^i)\tilde{\partial})(Q')=0.
$$
Any solution of the last equation {\it in the commutative ring} 
$\tilde{R}:=\tilde{K}[[x]]((\tilde{\partial}))$ has the form
$$
Q'= c_0 +c_1 \exp ((\xi_k-1)x\tilde{\partial}) +\ldots +c_{k-1}\exp ((\xi_k^{k-1}-1)x\tilde{\partial}), 
$$
where $c_i\in \tilde{R}$ don't depend on $x$ (as it follows from elementary differential algebra)\footnote{First note that $\ker\{\partial : \tilde{R}\rightarrow \tilde{R}\}=\tilde{K}((\tilde{\partial}))$. Suppose there is a solution $H$ such that it is linearly independent with $1,   \exp ((\xi_k-1)x\tilde{\partial}), \ldots , \exp ((\xi_k^{k-1}-1)x\tilde{\partial})$ over $\tilde{K}((\tilde{\partial}))$. Then $H_0:=\partial (H)=H'$ is not equal to zero and is linearly independent with  $\exp ((\xi_k-1)x\tilde{\partial}), \ldots , \exp ((\xi_k^{k-1}-1)x\tilde{\partial})$ over $\tilde{K}((\tilde{\partial}))$. By induction, $H_i:=(\partial+ (1-\xi_k^i)\tilde{\partial})(H_{i-1})=H_{i-1}'+ (1-\xi_k^i)\tilde{\partial} H_{i-1}$ is not equal to zero and is linearly independent with  $\exp ((\xi_k^{i+1}-1)x\tilde{\partial}), \ldots , \exp ((\xi_k^{k-1}-1)x\tilde{\partial})$ over $\tilde{K}((\tilde{\partial}))$ for all $i\le k-1$, in particular $H_{k-1}\neq 0$. On the other hand, $H_{k-1}=\prod_{i=1}^k(\partial+ (1-\xi_k^i)\tilde{\partial})(H)=0$, a contradiction.}.

If we choose a canonical representation form of elements in $\tilde{R}$ such that each monomial has the form $x^{j}\tilde{\partial}^{q}$ (with $x$ on the left and 
$\tilde{\partial}$ on the right), then the right hand side of the last formula can be rewritten as 
\begin{equation}
\label{E:summand'}
Q'':=c_0+\exp ((\xi_k-1)x\tilde{\partial})c_1+\ldots + \exp ((\xi_k^{k-1}-1)x\tilde{\partial})c_{k-1}=c_0+\tilde{A}_1c_1+\ldots + \tilde{A}_{k-1}c_{k-1}.
\end{equation}
Then $Q'=Q''$ also as elements written in this representation. Note that, since $Q'$ contains only non-negative powers of $\partial$, the series $Q''$ in \eqref{E:summand'} contains only non-negative powers of $\partial$  (and if we replace $\tilde{\partial}$ by $\partial$ in all terms of $Q''$, we get again the operator $Q$).  

\begin{lemma}
\label{L:various_decompositions}
For any $i=0, \ldots , k-1$ we have $\Ord (c_{i})\le \Ord (Q)$ in formula \eqref{E:summand'}, where the order $\Ord$ is defined on $\tilde{R}$ in the same way as on $\cm_1$.
\end{lemma}

\begin{proof}
Since the elements $c_{i}$, { being elements from $\tilde{K}((\tilde{\partial}))$, have {\it restricted} (positive) power in $\tilde{\partial}^{-1}$}, the expression in \eqref{E:summand'} can be written in the form
$$
(\tilde{c}_0+\tilde{A}_1\tilde{c}_1+\ldots + \tilde{A}_{k-1}\tilde{c}_{k-1})\tilde{\partial}^{-m},
$$
where $\tilde{c}_i\in \tilde{K}[[\tilde{\partial}]]$ and $m\ge 0$, i.e. the series in brackets is divisible by  $\tilde{\partial}^{m}$. We'll additionally assume that $m$ is minimal, i.e. $GCD(\tilde{c}_0, \ldots , \tilde{c}_{k-1})=1$ in the ring $\tilde{K}[[\tilde{\partial}]]$. 

Obviously, the homogeneous decomposition is unique also in the space $\tilde{R}$, and therefore in \eqref{E:summand'} we have the unique homogeneous decomposition 
$$
c_0+\tilde{A}_1c_1+\ldots + \tilde{A}_{k-1}c_{k-1}=\sum_{l\ge -m}(c_{0,l}+c_{1,l}\tilde{A}_1+\ldots + c_{k-1,l}\tilde{A}_{k-1})\tilde{\partial}^l,
$$
where $c_{i,j}\in \tilde{K}$. Since $\Ord (Q) <\infty$ and $\Ord (A_{k; i})=0$, we should have 
$$
\Ord (c_0+\tilde{A}_1c_1+\ldots + \tilde{A}_{k-1}c_{k-1}) \le \Ord (Q)
$$
and therefore by lemma \ref{T:A.Z L 7.2} $c_{i,l}=0$ for all $l> \Ord (Q)$ and all $i=0,\ldots ,k-1$. This means that $\Ord (c_i)\le \Ord (Q)$ for any $i$. 
\end{proof}

From lemma \ref{L:various_decompositions} it follows that the series $\tilde{c}_{i}$ will belong to $\hat{D}_1^{sym}$ after replacing $\tilde{\partial}$ by $\partial$ in all terms. Now note that in $\hat{D}_1^{sym}$ we have 
$$
(\tilde{c}_0+\tilde{A}_1\tilde{c}_1+\ldots + \tilde{A}_{k-1}\tilde{c}_{k-1})\tilde{\partial}^{-m}|_{\tilde{\partial}\mapsto \partial}=
(\tilde{c}_0+\tilde{A}_1\tilde{c}_1+\ldots + \tilde{A}_{k-1}\tilde{c}_{k-1})|_{\tilde{\partial}\mapsto \partial}\int^{m}
$$
i.e. $Q$ can be written in the form:
\begin{equation}
\label{E:summand}
Q=(\tilde{c}_0+\tilde{A}_1\tilde{c}_1+\ldots + \tilde{A}_{k-1}\tilde{c}_{k-1})|_{\tilde{\partial}\mapsto \partial}\int^{m}.
\end{equation}
Besides, all summands in the sum \eqref{E:summand} are well defined elements of $\hat{D}_1^{sym}$ of order $\le \Ord (Q)$ and their sum is also well defined in $\hat{D}_1^{sym}$. 

\begin{lemma}
\label{L:restriction_on_int}
In the formula \eqref{E:summand} we have $m\le k-1$. 
\end{lemma}

\begin{proof}
Assume the converse: $m\ge k$. By our assumption, the homogeneous component $Q_{-m}\neq 0$, and 
$$
Q_{-m}=(c_{0,-m}+c_{1,-m}{A}_1+\ldots + c_{k-1,-m}{A}_{k-1})\int^m
$$ 
Since $[\partial^k,Q]=0$, we have also 
$$
0=[\partial^k,Q_{-m}]=\partial^k(Q_{-m}) +\sum_{i=1}^{k-1}\binom{k}{i}\partial^{k-i}(Q_{-m})\partial^i.
$$
Let the canonical form of $Q_{-m}$ be $Q_{-m}=\sum_{p\ge 0}a_{m+p}x^{m+p}\partial^p$. Let $a_{m+z}$ be the first coefficient not equal to zero. Then we have
\begin{multline*}
\partial^k(Q_{-m}) +\sum_{i=1}^{k-1}\binom{k}{i}\partial^{k-i}(Q_{-m})\partial^i = 
a_{m+z}\frac{(m+z)!}{(m+z-k)!}x^{m+z-k}\partial^z + \sum_{p>0}a_{m+z+p}'x^{m+z-k+p}\partial^{z+p}\neq 0
\end{multline*}
(here $a_j'\in \tilde{K}$ are appropriate coefficients) -- a contradiction. 
\end{proof}

Now, expanding all brackets in \eqref{E:summand} and using the identities from item 1, we can rewrite $Q$ in the form stated in item 2. The uniqueness of coefficients follows immediately from lemma \ref{T:A.Z L 7.2}.

\end{proof}

Let's recall one notation and definition from \cite{A.Z}. For any $P\in \hat{D}_1^{sym}$ we put 
$$
P_{[q]}:= \frac{x^{q}}{q!}P_{(q)}, \mbox{ where}
$$ 
$$
P_{(q)} = q! \sum_{\substack{{k} \in \sdn_0 \\ {k} - {q} \le  \Ord (P)}} \alpha_{{k}, {q}} \, {\partial}^{{k}}, \quad \alpha_{{k}, {q}}\in K.
$$
The expression $P_{(q)}$ is called a {\it slice} and the sum  $P=\sum_{q\ge 0}P_{[q]}$ is called {\it a partial slice decomposition}.

Consider the space $F=K[\partial ]$. It has a natural structure of a right $\hat{D}_1^{sym}$-module via the isomorphism of vector spaces $F\simeq \hat{D}_1^{sym}/\idm \hat{D}_1^{sym}$.

\begin{Def}{(cf. \cite{A.Z}, Def. 6.4)}
\label{D:regular}
An element $P \in \hat{D}_1^{sym}$ is called \emph{regular} if the $K$--linear map $F \xrightarrow{ (-\circ \sigma(P))} F$ is injective, where $\circ$ means the action on the module $F$. In particular, $P$ is regular if and only if its symbol $\sigma(P)$ is regular.
\end{Def}

\begin{Prop}{(\cite{A.Z},Proposition 7.2)}
	\label{T:A.Z 7.2}
	Let $P\in  \hat{D}^{sym}_{1}$, $\Ord(P)=k>0$ be a regular operator. Then there exists an invertible operator $S\in  \hat{D}^{sym}_{1}$ with $\Ord (S)=0$ such that 
	$$P=S^{-1}\partial^{k}S $$
and $S_{[0]}=1$, $S_{[i]}=0$ for $0< i< k$. 
	
\end{Prop}

\begin{ex}
\label{Ex:normalised_dif_operators}
It's easy to see that an operator $P\in D_1$ with an invertible highest coefficient (cf.remark \ref{R:ord=deg}) is an example of a regular operator. 

Recall that any such operator can be normalised, i.e. reduced to the form $P=\partial^k + c_{k-2}\partial^{k-2}+\ldots + c_0$, with the help of some change of variables and conjugation by invertible function, see e.g. \cite[Prop. 1.3 and Rem. 1.6]{BZ}.
\end{ex}

\begin{Prop}
	\label{P:normalized_Schur}
	Let $P\in D_1$ be a normalized operator of positive degree, i.e. $P=\partial^k + c_{k-2}\partial^{k-2}+\ldots + c_0$, $k>0$. Then there exists an operator $S$ from proposition \ref{T:A.Z 7.2} such that $S_0=1$, $S_{-1}=0$. 
\end{Prop}

\begin{proof}
	The proof will follow the proof of Prop. 7.2 in \cite{A.Z}. $P$ is regular, since $\sigma (P)=\partial^k$.  Recall that there exists $S$ such that $S_{[0]}=1$, $S_{[i]}=0$ for $0< i < k$ and each slice can be found from the system
	$$
	(\partial^kS)_{[p]}=(SP)_{[p]}, \quad p\ge 0.
	$$ 
	Note that 
	$$
	(\partial^kS)_{[p]}=\partial^k(S_{[p+k]}) + \sum_{j=0}^{p+k-1}(\partial^kS_{[j]})_{[p]}
	$$ 
	and $(SP)_{[p]} = ((S_{[0]}+\ldots +S_{[p]})P)_{[p]}$. Then for $p=0$ we have
	$$
	(\partial^kS)_{[0]}=\partial^k(S_{[k]})+\partial^k=P_{[0]},
	$$
	and therefore the slice $S_{(k)}$ is uniquely determined; besides, $\Ord (S_{[k]})= \Ord (P- \partial^k)-k\le -2$. 
	
	Now we can use induction on $p$. By induction, we can assume $\Ord (S_{[k+j]})\le -2$ for all $j<p$.  
	Note that for $p>0$ we have $\Ord (P_{[p]})< k-2$ because $P$ is normalized. Then $\Ord ((SP)_{[p]})\le k-2$. On the other hand, we have $\Ord ((\partial^kS_{[j]})_{[p]})\le k-2$ for all $j<p+k$ (by induction). From the equation above the slice $S_{(k+p)}$ is uniquely determined as 
	$$
	\partial^k(S_{[k+p]})= (SP)_{[p]}-\sum_{j=0}^{p+k-1}(\partial^kS_{[j]})_{[p]},
	$$
	and therefore $\Ord (S_{[k+p]}) \le -2$. 
	
	Therefore, the homogeneous decomposition of $S$ has no homogeneous terms of order $-1$, and $S_{-1}=0$.
	
\end{proof}

\subsection{Basic  formulae in $\hat{D}_1^{sym}$}
\label{S:Basic  formulae in}

In this section we collect useful commutation relations between operators in the ring $\hat{D}_1^{sym}\hat{\otimes}_K\tilde{K}$ which will be used later. 

First we define a series of operators $B_i$:  $B_{1}:=\delta, B_{2}:=x\delta \partial ,\ldots, B_{n}:=\frac{1}{(n-1)!}x^{n-1}\delta \partial^{n-1}$, and  $B_j:=0$ for any $j\leq 0$. Define $\Gamma_i=(x\partial)^i$ for $i\ge 0$, and for $i<0, \Gamma_i=0$. For convenience, we introduce also a new notation: for any integer $n$ we set  $D^n=\partial^n$ if $n\ge 0$ and $D^n=\int^{-n}$  otherwise. Obviously, we have $\partial \delta=\delta x=0$ and $\Ord(B_{j})=0$ for any $j \in \mathbb{N}$.  From this observation and definition of $\delta$ we also get $\delta\delta=\delta$. 

\begin{lemma}
	\label{L:1}
	For a fixed $k\in \dn$ let $A_i=A_{k;i}$, $\xi$ be the $k$-primitive root, $B_j$ are defined as above. Then we have
	\begin{enumerate}
		\item $A_{i} B_{j}=B_{j} A_{i}=\xi^{i(j-1)}B_{j}$ for any $i,j \in \mathbb{N}$;
		\item 
	$$
\int x^m = \frac{m!}{(m+1)!}x^{m+1} + \sum_{i=1}^{\infty}(-1)^i \frac{m!}{(m+i+1)!}x^{m+i+1}\partial^i;
$$
	In particular, $\int x^m\delta = \frac{1}{m+1}x^{m+1}\delta$;
		\item $\int^{m} \partial^{m}=1-\sum_{k=1}^{m}B_{k}$ for any $m \in \mathbb{N}$;
		\item 
		$$\int^u f(x)=f(x)\int^u +
		\sum_{l=1}^{\infty}\binom{-u}{l}f(x)^{(l)}\int^{u+l}$$ for any  $f(x)\in \hat{R}$, $u\in\dn$;
		
		\item $B_i B_j=\delta_i^jB_j$, where $\delta_i^j$ is the Kronecker delta;
		
		\item $A_i\Gamma_j=\Gamma_jA_i$;
		\item $$
		D^i\Gamma_j=\sum_{l=0}^{j}\binom{j}{l}i^{j-l}\Gamma_lD^i, \quad \Gamma_jx^i=x^i(\sum_{l=0}^j\binom{j}{l}i^{j-l}\Gamma_l)
		$$
		\item $\Gamma_iB_j=B_j\Gamma_i=(j-1)^iB_j$;
		\item $D^u B_j=B_{j-u}D^u$. In particular, $D^u B_j=B_{j-u}D^u=0$ if $u>0$ and $j-u\le 0$ or $u<0$ and $j\le 0$.
	\end{enumerate}	
	where we assume in all these formulae that $0^0:=1$.
\end{lemma}

\begin{proof} 	
	1 can be directly calculated:
	$$
	A_{i}B_{j}=\frac{1}{(j-1)!}A_{i}x^{j-1}\delta \partial^{j-1}=\frac{1}{(j-1)!}\xi^{i(j-1)}x^{j-1}A_{i}\delta \partial^{j-1}=\frac{1}{(j-1)!}\xi^{i(j-1)}x^{j-1}\delta \partial^{j-1}=\xi^{i(j-1)}B_{j},
	$$ 
	$$
	B_{j}A_{i}=\frac{1}{(j-1)!}x^{j-1}\delta \partial^{j-1}A_{i}=\frac{1}{(j-1)!}\xi^{i(j-1)}x^{j-1}\delta A_{i}\partial^{j-1}=\frac{1}{(j-1)!}\xi^{i(j-1)}x^{j-1}\delta \partial^{j-1}=\xi^{i(j-1)}B_{j}.
	$$ 
	
	2. The proof is by induction on $m$. 
For $m=0$ it is true by definition of $\int$. For generic $m$ we have
\begin{multline*}
\int x^m = (\int x^{m-1})x= (\frac{(m-1)!}{m!}x^{m+1} + \sum_{i=1}^{\infty}(-1)^i \frac{(m-1)!}{(m+i)!}x^{m+i+1}\partial^i )+ \\
(\sum_{i=1}^{\infty}(-1)^i \frac{(m-1)!i}{(m+i)!}x^{m+i}\partial^{i-1})=
\frac{m!}{(m+1)!}x^{m+1} + \sum_{i=1}^{\infty}(-1)^i \frac{m!}{(m+i+1)!}x^{m+i+1}\partial^i.
\end{multline*}

	3. The proof is by induction on $m$.  When $m=1$, it is obvious. Suppose it has been done when $i<m$. Then
\begin{multline*}
	\int^{m} \partial^{m}=\int \int^{m-1} \partial^{m-1} \partial=\int(1-\sum_{k=1}^{m-1}B_{k})   \partial
\\  =\int(1-\sum_{k=1}^{m-2}B_{k})   \partial-\int(B_{m-1})  \partial=\int \int^{m-2} \partial^{m-2} \partial-\int(B_{m-1})  \partial
\\	=1-\sum_{k=1}^{m-1}B_{k}-\int (\frac{1}{(m-2)!}x^{m-2}\delta \partial^{m-2})\partial=1-\sum_{k=1}^{m-1}B_{k}-B_{m}.
\end{multline*}

	4. Note that the equality $\int x=x\int-\int^2$ holds iff $\int x\partial^q=(x\int-\int^2)\partial^q$ holds for any $q\ge 0$ (as it follows from definition of the ring $\hat{D}_1^{sym}$). Take $q=2$. Then we have 
	$$
\int x\partial^2=\int \partial^2x - 2\int \partial = (1-B_1)\partial x-2(1-B_1)= x\partial -1 +B_1
	$$
	On the other hand, 
	$$
(x\int-\int^2)\partial^2=x(1-B_1)\partial - (1-B_1-B_2)=x\partial -1+B_1.
	$$
	So, $\int x\partial^2=(x\int-\int^2)\partial^2$ and therefore $\int x=x\int-\int^2$. The second formula follows immediately by induction  for $f(x)=x^l$, and for generic series $f(x)=\sum_l c_lx^l$ the left and right hand side of the formula are just sums of homogeneous operators $\int^uc_lx^l$ (note that each such summand  is homogeneous of order $(-u-l)$ and therefore the total sum is well defined for any series $f(x)$)\footnote{This formula coincides with analogous formula for the classical ring of pseudo-differential operators, cf. e.g. \cite[Def. 4.1]{Zheglov_book}}. 
	
	5. Notice that $\partial^i x^j$ has a constant term  only when $i=j$. By this reason $B_iB_j=0$ if $i\neq j$, and 
$$
	B_iB_i=\frac{1}{(i-1)!}x^{i-1}\delta \partial^{i-1}\frac{1}{(i-1)!}x^{i-1}\delta \partial^{i-1}=\frac{1}{(i-1)!}x^{i-1}\delta\delta \partial^{i-1}=B_i .
$$
	
	6. $$
		A_i (x\partial)^j=(x\partial)A_i(x\partial)^{j-1}=s=(x\partial)^jA_i
		$$
	
	7. We have 
	\begin{multline*}
	\int(x\partial)=(\int x)\partial=(x\int-\int^2)\partial=x(1-\delta)-\int(1-\delta)
	=x-\int -x\delta+\int\delta=x-\int=(x\partial-1)\int .
	\end{multline*}
	Hence $$
		D^i(x\partial)=(x\partial+i)D^i, \quad (x\partial)x^i=x^i(x\partial+i). 
		$$
		and we have
		$$
		D^i\Gamma_j=(x\partial+i)^jD^i=\sum_{l=0}^{j}\binom{j}{l}i^{j-l}(x\partial)^lD^i, \quad \Gamma_jx^i=x^i(x\partial+i)^j=x^i(\sum_{l=0}^{j}\binom{j}{l}i^{j-l}(x\partial)^l).
		$$
	
	8. First note that
\begin{multline*}
	(x\partial)B_j=(x\partial)\frac{1}{(j-1)!}x^{j-1}\delta\partial^{j-1}
	=(j-1)\frac{1}{(j-1)!}x^{j-1}\delta\partial^{j-1}+\frac{1}{(j-1)!}x^{j}\partial\delta\partial^{j-1}
	=(j-1)B_j+0
\end{multline*}
	Analogously, $B_j(x\partial)=(j-1)B_j$. Hence 
	$$
		\Gamma_iB_j=(x\partial)^iB_j=(j-1)(x\partial)^{i-1}B_j=(j-1)^iB_j .
	$$
	
	9. Like before, notice that we have 
	$
	\partial B_j=B_{j-1}\partial 
	$
	and 
	$
	\int B_j=B_{j+1}\int .
	$
\end{proof}

The next Lemma can be deduced from the proof of \cite[Prop.7.1]{A.Z}. Another proof, by using our calculations above is as follows:

\begin{lemma}
	\label{L:homog_comm}
	Let $P\in \hat{D}_1^{sym}\hat{\otimes}_K\tilde{K}$ be a homogeneous operator commuting with $\partial^k$ and $-k<\Ord (P)=l<0$ (if $\Ord (P)\le -k$ then $P=0$ by proposition \ref{T:A.Z 7.1}, item 2). Then 
	$$
	P=(\sum_{j=0}^{k-1}c_jA_j)\int^{-l},
	$$
	where $c_j\in \tilde{K}$ satisfy the following conditions:
	$$
	\sum_{j=0}^{k-1}c_j=0, \quad \sum_{j=1}^{k-1}c_j (\xi^j-1)^q=0, \quad 1\le q\le -l-1,
	$$
	or, equivalently, 
	$$
	\sum_{j=0}^{k-1}c_j\xi^{j(q-1)}=0 \quad \mbox{for\quad } q=1, \ldots , -l.
	$$
	Vice versa, any such operator $P$ commutes with $\partial^k$.
\end{lemma}

\begin{proof}
	By proposition \ref{T:A.Z 7.1}, item 2 $P$ has the form as it is claimed and we only need to prove the relations between coefficients $c_j$. 
	Since the order of $P$ is negative, we have $P_{[q]}=0$ for $q=0, \ldots , -l-1$. On the other hand, using lemma \ref{L:1}, we have
	$$
	0=[P,\partial^{k}]= (\sum_{j=0}^{k-1}c_jA_j)(1-\sum_{q=1}^{-l}B_{q})\partial^{k+l}-(\sum_{j=0}^{k-1}c_jA_j)\partial^{k+l}= -\sum_{j=0}^{k-1}\sum_{q=1}^{-l}c_j\xi^{j(q-1)}B_q\partial^{k+l},
	$$
	hence $\sum_{j=0}^{k-1}c_j\xi^{j(q-1)}=0$ for $q=1, \ldots , -l$. 
	
	The same calculation proves the last assertion.
\end{proof}

\subsection{Homogeneous canonical polynomials}
\label{S:HCP}

In this section we fix some $k\in \dn$. Assume $\xi$ is the $k$-primitive root, $A_i=A_{k,i}$, $B_j,\Gamma_i$ are the same as in the previous section.  

The following definition is motivated by proposition \ref{T:A.Z 7.1} and some calculations below. 

\begin{Def}
\label{D:HCP}
Let $\tilde{K}=K[\xi]$. An element $H\in \hat{D}_{1}^{sym}\hat{\otimes}_K\tilde{K}$ is  called {\it homogeneous canonical polynomial  (in short of HCP)}  if $H$ can be written in the form
	\begin{equation}
	\label{E:HCP}
	H=[\sum_{0\leq i<k}f_{i;r}(x, A_{k;i}, \partial )+\sum_{0<j\leq N}g_{j;r}B_{j}]D^{r}
	\end{equation}
	for some $N\in\dn$, $r\in \dz$, where 
	
		1. $f_{i;r}(x, A_{k;i}, \partial)$ is a polynomial of $x, A_{k;i},\partial$,  $\Ord(f_{i;r})=0$,  of the form
		$$
		f_{i;r}(x, A_{k;i}, \partial )=\sum_{0\leq l\leq d_{i}}f_{l,i;r}x^l A_{k;i}\partial^l
		$$
		for some $d_i\in \dz_+$, where $f_{l,i;r}\in \tilde{K}$. The number  $d_i$ is called the {\it $x$-degree of $f_{i;r}$}:  $deg_{x}(f_{i;r}):= d_i$.
		
		2. $g_{j;r}\in \tilde{K}$.
		
		3. $g_{j;r}=0$ for $j\le -r$ if $r<0$.

In particular, $H$ is homogeneous and $\Ord (H)=r$. 
\end{Def}

Using the results of previous section, it can be shown that the form \eqref{E:HCP} of HCP from definition is uniquely defined. Namely, this follows from lemma:

\begin{lemma}
\label{L:correct_HPC}
Let $H$ be a HCP. Then $H=0$ iff $g_{j;r}=0$ for all $j$ and $f_{l,i;r}=0$ for all $i,l$. 

In particular, any HCP can be uniquely written in the form \eqref{E:HCP}.
\end{lemma}

\begin{proof}
Obviously, if all coefficients are equal to zero, then $H=0$. Now assume the converse:
$$
H=[\sum_{0\leq i<k}\sum_{0\leq l\leq d_{i}}f_{l,i;r}x^l A_{k;i}\partial^l+\sum_{0<j\leq N}g_{j;r}B_{j}]D^{r} =0
$$
and some coefficients $g_{j;r}, f_{l,i;r}$ are not equal to zero. Then necessarily $H':=H D^{-r}=0$.

Note that, using lemma \ref{L:1}, we can rewrite the first sum as
$$
\sum_{0\leq i<k}\sum_{0\leq l\leq d_{i}}f_{l,i;r}x^l A_{k;i}\partial^l=\sum_{0\leq i<k}\sum_{0\leq l\leq d_{i}}f_{l,i;r} \xi^{-il} A_{k;i} x^l \partial^l =\sum_{0\leq i<k}\sum_{0\leq l\leq d_{i}}f_{l,i;r}' \xi^{-il} A_{k;i} \Gamma_l, 
$$
where $f_{l,i;r}'\in \tilde{K}$ are some new coefficients, but $f_{d_i,i;r}'=f_{d_i,i;r}$ {\it for all $i$}. 

Next, note that 
$$
H'=[\sum_{0\leq i<k}\sum_{0\leq l\leq d_{i}}f_{l,i;r}' \xi^{-il} A_{k;i} \Gamma_l+\sum_{0<j\leq N'}g_{j;r}'B_{j}] 
$$
for some new $N'\in \dn$, $g_{j;r}'\in \tilde{K}$, but with {\it the same} coefficients $f_{l,i;r}'$. 
Indeed, if $r\ge 0$, then $D^r D^{-r}=1$ and therefore all coefficients of $H'$ are the same. If $r<0$, then by lemma \ref{L:1} $D^r D^{-r}=1-\sum_{i=1}^rB_i$, and by the same lemma any product 
$$
[\sum_{0\leq i<k}\sum_{0\leq l\leq d_{i}}f_{l,i;r}' \xi^{-il} A_{k;i} \Gamma_l+\sum_{0<j\leq N}g_{j;r}B_{j}]B_i
$$
is just a linear combination of some $B_j$. 

Let $d_{I}$ be a maximal $x$-degree, i.e. $f_{d_I,I;r}'$ is the highest non-zero coefficient. Note that for any $0\le t<k$ and for any $n\gg 0$ we have by lemma \ref{L:1}
\begin{multline*}
0=[\sum_{0\leq i<k}\sum_{0\leq l\leq d_{i}}f_{l,i;r}' \xi^{-il} A_{k;i} \Gamma_l+\sum_{0<j\leq N'}g_{j;r}'B_{j}]B_{kn+t+1}=
\sum_{0\leq i<k}\sum_{0\leq l\leq d_{i}} f_{l,i;r}' \xi^{-il} \xi^{it} (kn+t)^l B_{kn+t+1}
\end{multline*}
where
$$
\sum_{0\leq i<k}  f_{d_I,i;r}' \xi^{i(t-d_I)}=0, \quad 0\le t<k
$$
(where we assume $f_{d_I,i;r}'=0$ if $d_I>d_i$). But by the well known property of the Vandermonde matrix this system has the unique solution $f_{d_I,0;r}'=\ldots =f_{d_I,k-1;r}'=0$, a contradiction. So, 
$$
H=[\sum_{0<j\leq N}g_{j;r}B_{j}]D^{r}.
$$
Assume $g_{j_0;r}$ is the first non-zero coefficient. Then again by lemma \ref{L:1} $B_{j_0}H= g_{j_0;r}B_{j_0}D^{r}\neq 0$, a contradiction. 
\end{proof}

\begin{Def}	
\label{D:Sdeg_A}
Let $H$ be a HCP. We define 
$$Sdeg_A(H)=\max \{d_i|\quad 0\leq i<k \} \quad \mbox{or $-\infty$, if all $f_{l,i;r}=0$ } 
$$
and 
$$Sdeg_B(H)=\max\{j|\quad g_{j;r}\neq0\} \quad \mbox{or $-\infty$, if all $g_{j;r}=0$}
$$

We define a  {\it homogeneous canonical polynomial combination} (in short HCPC) as a finite sum of HCP (of different orders).
We extend the functions $Sdeg_A$, $Sdeg_B$ in an obvious way to all HCPCs. 

We'll say that a HCPC $H$ {\it doesn't contain $A_i$} if $f_{l,i;r}=0$ for all $i>0$ and all $r$. We'll say that a HCPC $H$ {\it doesn't contain $B_j$} if $Sdeg_B(H)=-\infty$.
\end{Def}

\begin{ex}
\label{Ex:HCP-dif-op}
Suppose $H$ is a HCP and $H\in D_1\hat{\otimes}_K\tilde{K}$ is a {\it differential operator}. Then it is easy to see that it can be uniquely written in the form \eqref{E:HCP}, which does not contain $A_i$ and $B_j$. 
\end{ex}

In practice it is more convenient to work with another form of HCPs, which we have already used in the proof of lemma \ref{L:correct_HPC}:  

\begin{Def}
\label{D:G-form}
	Suppose $H$ is a HCP. Then  $H$ can be (uniquely) written in another form: 
	$$
	H=(\sum_{0\leq i<k}\sum_{0\leq l\leq d_i} f'_{l,i;r}\Gamma_lA_i+\sum_{0<j\leq N}g_{j;r}B_{j})D^{r}
	$$
	which we'll call {\it the G-form of $H$}.
\end{Def}

It's easy to see that two forms of $H$ are one to one correspondence to each other, and therefore the G-form is also uniquely defined. Besides, the definitions of $Sdeg_A$, $Sdeg_B$ don't depend on the form, i.e. $Sdeg_A(H)$ is again the maximal $d_i$. 

The following lemma is the first obvious property of HCPCs.

\begin{lemma}
\label{L:obvious}
	Suppose $H$ and $M$ are two HCPCs, $k_1,k_2\in \tilde{K}$ are two arbitrary constant. Then $T=k_1H+k_2M$ is also a HCPC, with 
	$$
	Sdeg_A(T)\leq \max\{Sdeg_A(H), Sdeg_A(M) \}
	$$$$
	Sdeg_B(T)\leq \max\{Sdeg_B(H), Sdeg_B(M) \}
	$$
\end{lemma}

Before presenting next result, we need a lemma from standard ODE book:  

\begin{lemma}{(\cite[Ch.2, item 10 Th.8]{Pont})}
	\label{L:Port}
	Let $F$ be a field of characteristic 0. Suppose $L,f\in F[t]$ are non-zero polynomials with $\deg f=r$, $\lambda\in F$, and $q$ is the multiplicity of $t-\lambda$ in the polynomial $L(t)$ (if $L(\lambda )\neq 0$, then $q=0$). Then the ODE
	$$
	L(\partial )z=\sum_{i}a_{i}z^{(i)}=f(t)e^{\lambda t}
	$$
	has a solution in the form of 
	$$
	z_0=t^q g(t)e^{\lambda t},
	$$
	where $g(t)$ is  a polynomial of  degree $r$ (the same as $f$).
\end{lemma}

\begin{rem} Although this lemma was formulated and proved for the case $F=\dc$ in the book, its proof is valid in the case of arbitrary field of characteristic zero $F$ too. For, the claim of lemma is equivalent to the claim that the linear system on coefficients of the polynomial $g(t)$, obtained after substituting $z_0$ into the ODE, has a solution. Since it is linear, its solvability  does not depend on the ground field.
\end{rem}

Now we are ready to prove the following claim. 

\begin{lemma}
	\label{L:ODE sol}
	Suppose $M$ is a HCPC. Suppose $H\in \hat{D}_1^{sym}\hat{\otimes}_K\tilde{K}$ is a HCPC  satisfying the condition
	$$
	[\partial^k,H]=M
	$$
	where $k$ is the original $k$ we have fixed. Then we have 
	\begin{enumerate}
	    \item If $Sdeg_A(M)\neq -\infty$ then $Sdeg_A(H)=Sdeg_A(M)+1$ and $Sdeg_B(H)=Sdeg_B(M)$.
	    \item If $Sdeg_A(M)= -\infty$ then  $Sdeg_A(H)$ is either $0$ or $-\infty$ and $Sdeg_B(H)=Sdeg_B(M)$.
	\end{enumerate}
\end{lemma}

\begin{proof} 
	1.  Consider the HCP decomposition of $M$: $M=\sum_{i=l_1}^{l_2}M_i$, where $\Ord(M_i)=i$, and the HCP decomposition of $H$. One can easily see that
	$$
	[\partial^k,H_{i-k}]=M_i
	$$
	holds for any $i$. Since $Sdeg_A,Sdeg_B$ of $H,M$ are defined as the maximum values of $Sdeg_A,Sdeg_B$ on $H_i,M_j$,  we can assume that $H,M$ are HCPs. 
	
	2. We use the same idea as in the proof of proposition \ref{T:A.Z 7.1}. Considering $\partial$ as a constant the original equation becomes an ODE in the ring $\tilde{R}$: 
		\begin{equation}
		H^{(k)}+\dots+k H'\tilde{\partial}^{k-1}=M.
		\end{equation}	
        The eigen-polynomial of this ODE is $L(t):=(t+\tilde{\partial})^k-\tilde{\partial}^k$. Any solution $H$ of this equation looks like $H=H_0+\sum_{i=0}^{k-1}c_i\tilde{A}_{k,i}$, where $c_i\in \tilde{K}((\tilde{\partial}))$ and $H_0$ is a special solution. Since $M$ is a HCP, it can be represented as a linear combination of quasi-polynomials  $f_i(x)\tilde{A}_i\tilde{\partial}^r$ and $\tilde{B}_j\tilde{\partial}^r$, where $f_i(x)=\sum_l f_{l,i}x^l\tilde{\partial}^l$ (see the definition of HCP) in the ring $\tilde{R}$, which are homogeneous with respect to the order $\Ord$. Therefore, there is a partial solution equal to a sum of partial solutions of similar equations with monomial right hand side $M$.   
       
       3. We have two possible cases of such monomials:
        \begin{enumerate}
        	\item $M=f_i(x)\tilde{A}_i{ \tilde{\partial}^{\Ord(M)}}$. Notice that $\tilde{A}_i=e^{(\xi^i-1)x\tilde{\partial}}$ and $(\xi^i-1)\tilde{\partial}$ is the root of $L(t)$ with multiplicity 1. Then according to  Lemma \ref{L:Port}, such ODE has a special solution $H_0=xg_i(x)\tilde{A}_i{ \tilde{\partial}^{\Ord(M)}}$ (with $deg_x(f_i)=deg_x(g_i)$). 
 
 Now note that, since $M$ is homogeneous, a special solution can be chosen to be also homogeneous (just take the homogeneous component of any special solution), because the left hand side of our ODE is homogeneous for any homogeneous $H$. Then such a solution will be also an HCP and  
        	$$
        	Sdeg_A(H_0)=Sdeg_A(M)+1, \quad Sdeg_B(H_0)=Sdeg_B(M)=-\infty .
        	$$
        	\item $M=b_j\tilde{B}_j{ \tilde{\partial}^{\Ord(M)}}$, where $b_j\in \tilde{K}$. Notice that $\tilde{B}_j=cx^{j-1}\tilde{\partial}^{j-1}\tilde{\delta}= cx^{j-1}\tilde{\partial}^{j-1}e^{-x\tilde{\partial}}$, and $-\tilde{\partial}$ is not a root of $L(t)$. Then according to Lemma \ref{L:Port}, we know the special solution of this ODE will be in the form 
        	$$
        	H_0=h_j(x)\tilde{\delta}{\tilde{\partial}^{\Ord(M)}}
        	$$
        	with $deg_x(h_j)=j-1=deg_x(B_j)$. Again a special solution can be chosen to be homogeneous, i.e. it is a HCP. Thus we can write this $H_0$ in the form 
        	$$
        	H_0=\sum_{l=1}^{j}h_{l-1,j}\tilde{B}_l\tilde{\partial}^{\Ord (M)-k}
        	$$
        	with $h_{l-1,j}\in \tilde{K}$, $h_{j-1,j}\neq 0$. 
        	Hence $Sdeg_A(H_0)=-\infty$, $Sdeg_B(H_0)=Sdeg_B(M)$. 
       \end{enumerate}
        
        4. Now we have: $H=H_0+\sum_{i=0}^{k-1}c_i\tilde{A}_{k,i}$, where $c_i\in \tilde{K}((\tilde{\partial}))$ and $H_0$ is a HCP. Moreover, if  $Sdeg_A(M)\neq -\infty$, then $Sdeg_A(H_0)=Sdeg_A(M)+1$ and $Sdeg_B(H_0)=Sdeg_B(M)$. And if $Sdeg_A(M)= -\infty$, then $Sdeg_A(H_0)= -\infty$  and $Sdeg_B(H_0)=Sdeg_B(M)$. Since our original solution $H\in \hat{D}_1^{sym}\hat{\otimes}_K\tilde{K}$, we have $\Ord (H)<\infty$ and therefore $\Ord (c_i)< \infty$ for all $i$, i.e. $c_i\in \tilde{K} [\tilde{\partial}^{-1}, \tilde{\partial}]$ is a monomial since $H,H_0$ are a HCPs.  
        
Writing this solution in a canonical form and replacing $\tilde{\partial}$ by $D$, we get, as in the proof of propostion \ref{T:A.Z 7.1}, the original solution $H$. Since $Sdeg_A(\sum_{i=0}^{k-1}c_i{A}_{k,i})=0$, and in view of lemma \ref{L:correct_HPC}, we get our assertion. Thus we complete the proof.
	
\end{proof}

\begin{rem}
\label{R:ODE sol}
The condition about existence of $H$ from lemma is essential: for example $[\partial^4,H]=B_{10}\partial^4$ doesn't have any HCPC solution. But we will see if $M$ doesn't contain $B_j$ such $H$ must exist, see lemma \ref{L:ODE sol 2}.
\end{rem}

The following lemma describes basic properties of functions $Sdeg_A$ and $Sdeg_B$. Besides, it contains also useful formulae for monomial multiplication. 

\begin{lemma}
	\label{L:HCPC}
	Suppose $H,M$ are two  HCPCs. Then $T:=HM$ is also a HCPC, what's more we have 
	\begin{enumerate}
		\item $Sdeg_A(T)\leq Sdeg_A(H)+Sdeg_A(M)$ (here we assume $-\infty +n=-\infty$). 
		\item 
		if $H,M$ are two  HCPs then
		\begin{enumerate}
			\item If $\Ord(H)\geq 0$, then $Sdeg_B(T)\leq \max \{Sdeg_B(H),Sdeg_B(M)\}$
			\item If $\Ord(H)< 0$, then $Sdeg_B(T)\leq \max \{Sdeg_B(H),Sdeg_B(M)-\Ord(H),-\Ord(H)\}$
		\end{enumerate}
	\end{enumerate}
\end{lemma}

\begin{proof}
	First let's prove $T$ is a HCPC. Obviously, it is enough to prove this for the case when $H, M$ are monomials written in the G-form. As a byproduct we'll get convenient multiplication formulae of HCPCs.
	
	Consider the following 4 cases (below we assume in all formulae that $0^0:=1$):
	
		1. $H=b_iB_iD^u, M=c_jB_jD^v$, where $i\ge 1-u$ if $u<0$ and $j\ge 1-v$ if $v<0$, $b_i,c_j\in \tilde{K}$. We have by lemma \ref{L:1}
		\begin{equation}
		\label{E:BB}
		HM=b_iB_iD^uc_jB_jD^v=b_ic_jB_iB_{j-u}D^uD^v= b_ic_j\delta_i^{j-u}B_iD^{u+v} 
		\end{equation}
(here $\delta_i^{j-u}$ is the Kronecker delta), because if $i=j-u$, then $i-1+u+v=j-1+v\ge 0$ and $B_iD^{u+v}\neq 0$.
		
		2. $H=b_iB_iD^u$, where $i\ge 1-u$ if $u<0$, $M=a_{l,m}\Gamma_m A_jD^v$, $b_i, a_{l,m}\in \tilde{K}$. By Lemma \ref{L:1} item 9 we know $B_iD^u=D^uB_{i+u}$, hence $i\ge 1-u$, otherwise $H=0$. Then we have
		\begin{equation}
		\label{E:BA}
		HM=\begin{cases}
			0 & i-1+u+v < 0
		\\a_{l,m}b_i\xi^{j(u+i-1)}(i-1+u)^m B_iD^{u+v} &\text{Otherwise}
		\end{cases}
		\end{equation}
		This is because 
		\begin{multline*}
		B_iD^u\Gamma_mA_jD^v=B_i(x\partial+u)^mD^uA_jD^v=\xi^{uj}B_i(x\partial+u)^mA_jD^uD^v
		\\=\xi^{uj}B_iA_j(x\partial+u)^mD^uD^v=\xi^{(u+i-1)j}B_i(x\partial+u)^mD^uD^v
		=\xi^{(u+i-1)j}(i-1+u)^mB_iD^uD^v
		\end{multline*}
		where the first equality is by Lemma \ref{L:1} item 7; the second equality is by Prop \ref{T:A.Z 7.1}; the third equality is by Lemma \ref{L:1} item 6; the forth equality is by Lemma \ref{L:1} item 1; the fifth equality is by Lemma \ref{L:1} item 8, since $B_i(x\partial)=(i-1)B_i$, so $B_i(x\partial+u)=(u+i-1)B_i$. Notice that when $u<0$, and $v>0$, we have 
		$$
		D^uD^v=(1-\sum_{s=-u-v+1}^{-u}B_s)D^{u+v}
		$$
		And by item 1 we know $B_iB_j=\delta^i_jB_i$, hence we know if $-u-v+1\leq i\leq -u$, then 
		$$
		B_iD^uD^v=(B_i-B_i)D^{u+v}=0
		$$
		But we already assume $i\ge 1-u$, hence $B_iD^uD^v=B_iD^{u+v}$, again by Lemma \ref{L:1} item 9, we know it is 0 when $i-1+u+v<0$.
		
		3. $H=a_{l,m}\Gamma_m A_jD^v, M=b_iB_iD^u$, where $i\ge 1-u$ if $u<0$, $b_i, a_{l,m}\in \tilde{K}$. We have
		\begin{equation}
		\label{E:AB}
		HM= a_{l,m} b_i \xi^{j(i-v-1)}(i-v-1)^m B_{i-v}D^vD^u= \begin{cases}
			0 & i-v< 1
		\\ \lambda B_{i-v}D^{u+v} &\text{Otherwise.}
		\end{cases}
		\end{equation}
where $\lambda =	a_{l,m} b_i \xi^{j(i-v-1)}(i-v-1)^m$. This is because when $i-v< 1$ we have $D^vB_i=B_{i-v}D^v=0$ (c.f. Lemma \ref{L:1} item 9 and by assumptions for $r\leq 0$, $B_r=0$), so that 
$$
\Gamma_m A_jD^vB_iD^u=\Gamma_m A_j(D^vB_i)D^u=0
$$
and when $i-v\ge 1$, we have 
\begin{multline*}
	\Gamma_m A_jD^vB_iD^u=\Gamma_mA_jB_{i-v}D^vD^u=\xi^{j(i-v-1)}\Gamma_mB_{i-v}D^vD^u
	\\=\xi^{j(i-v-1)}(i-v-1)^mB_{i-v}D^vD^u=\xi^{j(i-v-1)}(i-v-1)^mB_{i-v}D^{v+u}
\end{multline*}
The last equality holds obviously when $v\ge 0$ or $u\leq 0$. When $v<0,u>0$, notice that $i>0$, so that $i-v>-v$, hence for any $-u-v+1\leq s\leq -v$, we have $B_{i-v}B_s=\delta_{s}^{i-v}=0$ by Lemma \ref{L:1} item 5. Hence
$$
B_{i-v}D^vD^u=B_{i-v}(1-\sum_{s=-u-v+1}^{-v}B_s)D^{u+v}=B_{i-v}D^{u+v}
$$
		
		4. $H=a_{i,m}\Gamma_mA_iD^u, M=a_{j,n}\Gamma_nA_jD^v$, $a_{i,m}, a_{j,n}\in \tilde{K}$.  We have 
		\begin{multline}
		\label{E:AA}
		HM= a_{i,m} a_{j,n} \xi^{uj} \sum_{l=0}^{n}\binom{n}{l}u^{n-l}\Gamma_{l+m}A_{i+j}D^uD^v
		\\=\begin{cases}
		a_{i,m} a_{j,n} \xi^{uj} \sum_{l=0}^{n}\binom{n}{l}u^{n-l}\Gamma_{l+m}A_{i+j}D^{u+v}+\sum_{s=-u-v+1}^{-u}\lambda_sB_sD^{u+v} &u<0,v>0
		\\ a_{i,m} a_{j,n} \xi^{uj}\sum_{l=0}^{n}\binom{n}{l}u^{n-l}\Gamma_{l+m}A_{i+j}D^{u+v} &\text{Otherwise.}
		\end{cases}
		\end{multline}
		where $\lambda_s=-a_{i,m}a_{j,n}\xi^{u j+(i+j)(s-1)}(u+s-1)^n(s-1)^m$, when $$\max\{1,-u-v+1 \}\leq s\leq -u.$$ 
		This is because
		\begin{multline*}
			\Gamma A_i D^u\Gamma_nA_jD^v=\Gamma_mA_i(x\partial+u)^uD^uA_jD^v=\xi^{uj}A_i\Gamma_m(x\partial+u)^nA_jD^uD^v
			=\xi^{uj}A_{i+j}\Gamma_m(x\partial+u)^nD^uD^v
		\end{multline*}
		and for $-u-v+1\leq s\leq -u$ when $u<0,v>0$, we have 
		\begin{multline*}
		A_{i+j}\Gamma_m(x\partial+u)^nB_s=(s-1+u)^nA_{i+j}\Gamma_mBs
		\\=(s-1)^m(s-1+u)^nA_{i+j}B_s=\xi^{(i+j)(s-1)}(s-1)^m(s-1+u)^nB_s
		\end{multline*}
	
    Hence we know the product is a HCPC. Now we can estimate the $Sdeg$ of $T$:
    
    	1. Consider $Sdeg_A(T)$. Observe that in cases 1-4 our claim is true. So, it is true for a product of any two HCPCs. Note also that in the case 4, for $H,M$ being monomials, we have the equality $Sdeg_A(T)=Sdeg_A(H)+Sdeg_A(M)$, however, in general, if we take the product of sums of monomials, a strict inequality can appear. 
    	 
    	2. Consider $Sdeg_B(T)$. Then explicit formulae of cases 1-4 show our assertion in general case.  
    
\end{proof}

\begin{cor}
\label{C:HCPC(k)}
	For a fixed $k$ all $HCPCs$  form a subring in the ring $\hat{D}_1^{sym}\hat{\otimes}_K\tilde{K}$.
\end{cor}

\begin{rem}
Denote by
$$
Hcpc(k):=\{\text{all HCPCs assoicated to }k\}
$$
the subring from corollary. 
It's easy to observe that if $a|b$, $b=ra$ and $\xi,\eta$ are respectively  the $a,b$-primitive roots, then $\eta^r=\xi$, i.e. $A_{b,ir}=A_{a,i}$. This means
$$
Hcpc(a)\subseteq Hcpc(b)
$$
Thus for $p,q$, if we assume $r=gcd(p,q)$, then we have 
$$
Hcpc(p)\bigcap Hcpc(q)\supseteq Hcpc(r).
$$
\end{rem}

\begin{lemma}
	\label{L:ODE sol 2}
	Suppose $M$ is a HCPC which doesn't contain $B_j$ (i.e. $Sdeg_B(M)=-\infty$). Then there exists a HCPC $H\in \hat{D}_1^{sym}\hat{\otimes}_K\tilde{K}$, which satisfies the condition
	$$
	[\partial^k,H]=M. 
	$$
	\end{lemma}

\begin{proof}
	Take the HCP decomposition of $M$ $M=\sum_{i=l_1}^{l_2}M_i$ , where $\Ord(M_i)=i$. 
	
It's enough to find HCPs $H_i$ such that 
	$$
	[\partial^k,H_i]=M_{i+k}.
	$$
	From now on, we always assume both $H,M$ are HCPs, with $\Ord(H)=m-k$, $\Ord(M)=m$.
	
	Written $H,M$ in G-form, we solve the equation into two steps:
	
		1. Assume 
		$$
		H=(H_{0;m-k}+H_{1;m-k}A_1+\dots+H_{k-1;m-k}A_{k-1})D^{m-k}
		$$
		with 
		$$
		M=(M_{0;m}+M_{1;m}A_1+\dots+M_{k-1;m}A_{k-1})D^m.
		$$
Since $[\partial^k, A_i]=0$, we know by lemma \ref{L:HCPC}  that 
		\begin{equation}
		\label{E:(*)}
		[\partial^k,H_{i;m-k}A_i D^{m-k}]=M_{i;m}A_iD^m \qquad [\mod B_j],
		\end{equation}
	where $[\mod B_j]$ denote denote terms containing $B_j$s. 
		Now assume 
		$$
		H_{i;m-k}=\sum_{l=0}^{t_i}h_{l,i;m-k}\Gamma_l.
		$$
		Since $[\partial^k,h_{0,i;m-k}A_iD^{m-k}]$ contains only terms with $B_j$,  we can ignore terms with $h_{0,i;m-k}$ by looking for a solution. Assume 
		$$
		M_{i;m}=\sum_{l=0}^{t_i-1}m_{l,i;m}\Gamma_l. 
		$$
		Using  Lemma \ref{L:HCPC}, we can directly calculate $[\partial^k,H_{i;m-k}A_i D^{m-k}]$. Then we get the linear system for the unknown coefficients ${h_{1,i;m-k},\dots,h_{t_i,i;m-k}}$:
		\begin{equation}
		\label{E:matrix_equation}
		\begin{pmatrix}
			 k & * & \dots & * \\
			0 & 2 k & \dots & * \\
			\dots & \dots & \dots & \dots \\
			0 & 0 & \dots & t_i k
		\end{pmatrix}\begin{pmatrix}
		h_{1,i;m-k} \\
		h_{2,i;m-k} \\
		\dots \\
		h_{t_i,i;m-k}
		\end{pmatrix}=\begin{pmatrix}
		m_{0,i;m} \\
		m_{1,i;m} \\
		\dots \\
		m_{t_i-1,i;m}
		\end{pmatrix}
		\end{equation}
		Since  this is a lower-triangular matrix, with all diagonal elements non-zero, this equation system always has  a solution. Solving it, we get almost all coefficients in $H$ except for $h_{0,0;m-k},\dots,h_{0,k-1;m-k}$.
		
		2. Suppose $\tilde{H}$ is the result we get in step 1. From the discussion in step 1, we know 
		$$
		[\partial^k,\tilde{H}]=M \quad [\mod B_j]
		$$
		Now assume  $$
		\bar{H}=(h_{0,0;m-k}+\cdots+h_{0,k-1;m-k}A_{k-1})D^{m-k}
		$$
		with $H=\tilde{H}+\bar{H}$. The equation $[\partial^k,H]=M$ becomes
		\begin{equation}
			\label{E:Bj part in ODEsolution}
		[\partial^k,\bar{H}]=M-[\partial^k,\tilde{H}]
	    \end{equation}
		Notice that terms on the right hand side are already known, and they contains only terms with $B_j$ (i.e. $Sdeg_A=-\infty$). There are three possible cases:
		\begin{enumerate}
			\item $m\ge k$, in such case no $B_j$ can appear on the right hand side of \eqref{E:(*)}. So we can simply put $h_{0,0;m-k},\dots,h_{0,k-1;m-k}$ all 0. We will get $[\partial^k,H]=M$.
			\item $m\leq 0$, in such case, notice that 
			$$
		    D^{m-k}\partial^k=\int^{-m+k}\partial^k=(1-B_{-m+1}-\cdots-B_{-m+k})
			$$
			So there might be $k$ summands  of $B_j$  in the left hand side of equation \eqref{E:Bj part in ODEsolution}. On the other hand, we know there might be at most $k$ summands with $B_j$ ($B_{-m+1},\dots,B_{-m+k}$) on the right hand side of the equation (according to the uniqueness of HCPC). By the same reason  we know the coefficients at $B_{-m+1},\dots,B_{-m+k}$ on both sides must be equal to each other respectively. Hence we get $k$ linear equations for $h_{0,0;m-k},\dots,h_{0,k-1;m-k}$.
			
			Calculating the coefficients at $B_{-m+j}$ on both sides, we have 
			$$
			\sum_{i=0}^{k-1}\xi^{i(j-m-1)}h_{0,i;m-k}=b_{j-m}
			$$
			where $b_{1-m},\dots,b_{k-m}$ are already known on the right hand side. Collecting this linear system, we find we have $k$ variables and $k$  equations, the coefficient matrix is a Vandermonde matrix (hence always of full rank). Thus we can always solve the equation and find the coefficients $h_{0,0;m-k},\dots,h_{0,k-1;m-k}$. After that we get $[\partial^k,\tilde{H}+\bar{H}]=M$.
			\item $0<m<k$, in this case the same arguments as in case b) work, and 
			we omit the details. We only need to notice that here we have $k$ variables but $(k-m)$ equations, and the coefficient matrix becomes the full rank sub-matrix of the Vandermonde matrix. So the solution exists but might be not unique.  
		\end{enumerate}
	
\end{proof}

\subsection{Some necessary conditions on the Schur operator}
\label{S:necessary conditions on the Schur}

Let $Q\in D_1$ be a normalized differential operator, i.e $Q=\partial^q+(\dots)\partial^{q-2}+\dots$. According to Propositions \ref{T:A.Z 7.2}, \ref{P:normalized_Schur} there exists $S\in \hat{D}_{1}^{sym}$, such that $S^{-1}QS=\partial^q$, where $S=S_{0}+S_{-1}+\dots$, with $S_{0}=1,S_{-1}=0$. 

In this section we establish several necessary conditions on $S$. Namely, we'll show that all homogeneous components of $S$, $S^{-1}$ are HCPs with $Sdeg_B=-\infty$, and the function $Sdeg_A$ has a linear upper and lower bound. 

From now on we fix $k=q=\Ord(Q)$ (recall that in our case $\Ord(Q)=deg (Q)$). Let $\xi$ be a $q$-th primitive root of $1$, $\tilde{K}=K[\xi ]$.

\begin{Def}
\label{D:total_B}
Let $H$ be an element from $Hcpc(k)$. We'll say $H$ is {\it totally free of $B_j$} if $Sdeg_B(HD^p)=-\infty$ for all $p\in\dz$\footnote{Note that definition of $Sdeg_B$ is sensitive with respect to the multiplication by $D^p$. For example, $H=\int$ is a HCP with $Sdeg_B(H)=-\infty$, but $HD= 1-\delta$ is a HCP with $Sdeg_B(HD)=1$. In this example $H$ doesn't contain $B_j$, but is not totally free of $B_j$.}. 
\end{Def}

\begin{ex}
\label{Ex:dif_op_total_B}
It's easy to see that an operator $P\in D_1$ with an invertible highest coefficient (cf.remark \ref{R:ord=deg}) written in G-form is totally free of $B_j$, i.e. all its homogeneous components have this property. 
\end{ex}

\begin{lemma}
\label{L:total_B}
The subset of totally free of $B_j$ elements from $Hcpc(k)$ for a fixed $k$ form a subring in $\hat{D}_1^{sym}\hat{\otimes}_K\tilde{K}$. 
\end{lemma}

\begin{proof}
Clearly, this subset form a linear subspace (cf. lemma \ref{L:obvious}). To prove the claim it suffices to prove that the product of two totally free of $B_j$ HCPs $H_1, H_2$ is totally free of $B_j$. 

Assume $H_1=(\sum_{0\leq i<k}\sum_{0\leq l\leq d_i} f'_{l,i;r}\Gamma_lA_i)D^u$ is written in the G-form. 
Since $H_2$ is totally free of $B_j$, the G-form of $H_2D^{p}$ does not contain $B_j$ for any $p\in \dz$. Therefore, to show that $H_1H_2$ is totally free of $B_j$ it suffices to show that the product of $H_1$ with any monomial $c\Gamma_{n}A_{q}D^{v}$ does not contain $B_j$. 

By lemma \ref{L:1} we have 
$$
H_1(c\Gamma_{n}A_{q}D^{v})= c(x\partial +u)^n\xi^{uq}A_q(H_1D^v),
$$
and $H_1D^v$ does not contain $B_j$, hence $H_1(c\Gamma_{n}A_{q}D^{v})$ does not contain $B_j$ and we are done.
\end{proof}

\begin{Prop}
\label{P:lower estimate}
Suppose 
$$
0\neq H=\sum_{i=0}^{k-1}\sum_{m=0}^da_{m,i}A_i\Gamma_m\int^u.
$$ 	
is a HCP  from $Hcpc(k)$,  with $\Ord(H)=-u<0$. Then $H$ is totally free of $B_j$ iff the linear system of equations on $a_{m,i}$ holds:
\begin{equation}
\label{E:system_on_B_j}
\sum_{i=0}^{k-1}\sum_{m=0}^d\xi^{i(j-1)}(j-1)^ma_{m,i}=0
\end{equation}
for any $1\leq j\leq u$.

Moreover, if $H$ is totally free of $B_j$, then $\frac{u}{k}-1< Sdeg_A(H).$
\end{Prop}

\begin{proof}
    Suppose  $H$ is totally free of $B_j$. Since 
\begin{multline*}
	H\partial^u=\sum_{i=0}^{k-1}\sum_{m=0}^da_{m,i}A_i\Gamma_m(1-B_1-\cdots-B_{u})
	=\sum_{i=0}^{k-1}\sum_{m=0}^da_{m,i}A_i\Gamma_m-\sum_{j=1}^{u}\sum_{i=0}^{k-1}\sum_{m=0}^d\xi^{i(j-1)}(j-1)^ma_{m,i}B_j
\end{multline*}
is free of $B_j$, we get a linear equation system \eqref{E:system_on_B_j}. Vice versa, this system implies the total freeness of $B_j$: indeed, for any $l\ge 0$ $H\partial^l= H\partial^u D^{l-u}$, hence it is free of $B_j$ (for $l<0$  no terms with $B_j$ can appear). 

The second statement requires more technical arguments.

\begin{lemma}
	\label{L:Tkdu full rank when u=(d+1)k}
	Suppose $T$ is the coefficient matrix of  system \eqref{E:system_on_B_j}. If $u=(d+1)k$, $T$ is a square matrix and we have $det(T)\neq 0$.
\end{lemma}

\begin{proof}
	Note that $T$ is in the shape of 
	$$
	\begin{pmatrix}
		1 & 1 & 1 & \cdots & 1 \\
		1 & \xi & \xi^2 & \cdots & \xi^{u-1} \\
		\cdots & \cdots & \cdots & \cdots & \cdots \\
		1 & \xi^{k-1} & \xi^{2(k-1)} & \cdots & \xi^{(k-1)(u-1)} \\
		\cdots & \cdots & \cdots & \cdots & \cdots \\
		0 & 1 & 2^d & \cdots & (u-1)^d \\
		0 & \xi & \xi^22^d & \cdots & \xi^{u-1}(u-1)^d \\
		\cdots & \cdots & \cdots & \cdots & \cdots \\
		0 & \xi^{k-1} & \xi^{2(k-1)}2^d & \cdots & \xi^{(k-1)(u-1)}(u-1)^d
	\end{pmatrix}
	$$
	suppose there exist $\overrightarrow{\alpha}=(\alpha_1,\dots,\alpha_{(d+1)k}),\alpha_l\in \tilde{K}, 1\leq l\leq k(d+1)$, such that 
	$$
	\overrightarrow{\alpha}T=0
	$$
	Then we have 
	\begin{multline*}
	(\alpha_{(d+1)k}\xi^{(k-1)(u-1)}+\cdots+\alpha_{(d+1)k-k+1})(u-1)^d+(\cdots)(u-1)^{d-1}
	+\cdots+(\alpha_k\xi^{(k-1)(u-1)}+\cdots+\alpha_1)=0
	\end{multline*}
	and 
	\begin{multline*}
		(\alpha_{(d+1)k}\xi^{(k-1)(u-2)}+\cdots+\alpha_{(d+1)k-k+1})(u-2)^d+(\cdots)(u-2)^{d-1}
		+\cdots+(\alpha_k\xi^{(k-1)(u-2)}+\cdots+\alpha_1)=0
	\end{multline*}
	\dots
\begin{multline*}
	(\alpha_{(d+1)k}\xi^{k-1}+\cdots+\alpha_{(d+1)k-k+1})+(\cdots)
	+\cdots+(\alpha_k\xi^{k-1}+\cdots+\alpha_1)=0
\end{multline*}

Now denote 
$$
\begin{cases}
\eta_0=\alpha_{(d+1)k}\xi^{(k-1)(u-1)}+\cdots+\alpha_{(d+1)k-k+1}
\\\eta_1=\alpha_{(d+1)k-k}\xi^{(k-1)(u-1)}+\cdots+\alpha_{(d+1)k-2k+1}
\\\cdots
\\\eta_d=\alpha_k\xi^{(k-1)(u-1)}+\cdots+\alpha_1
\end{cases}
$$
Notice that $\xi^k=1$, we have 
$$
\alpha_{(d+1)k}\xi^{(k-1)(u-k-1)}=\alpha_{(d+1)k}\xi^{(k-1)(u-1)}
$$
Hence we know 
$$
\begin{cases}
	(u-1)^d\eta_0+(u-1)^{d-1}+\cdots+\eta_d=0
	\\(u-k-1)^d\eta_0+(u-k-1)^{d-1}+\cdots+\eta_d=0
	\\\cdots
	\\(u-dk-1)^d\eta_0+(u-dk-1)^{d-2}\eta_1+\cdots+\eta_d=0
\end{cases}
$$
Considering $\eta_0,\eta_1,\dots,\eta_d$ as variables, we get a linear equation systems equipped with Vandermonde's coefficient matrix. Hence we know $\eta_0=\eta_1=\cdots=\eta_d=0$. In the same way, if we denote
$$
\begin{cases}
	\tilde{\eta}_0=\alpha_{(d+1)k}\xi^{(k-1)(u-2)}+\cdots+\alpha_{(d+1)k-k+1}
	\\\tilde{\eta}_1=\alpha_{(d+1)k-k}\xi^{(k-1)(u-2)}+\cdots+\alpha_{(d+1)k-2k+1}
	\\\cdots
	\\\tilde{\eta}_d=\alpha_k\xi^{(k-1)(u-2)}+\cdots+\alpha_1
\end{cases}
$$
we will also get $\tilde{\eta}_0=\tilde{\eta}_1=\cdots=\tilde{\eta}_d=0$. We play with this game for $k$ times, there will be $k$ numbers of equation for $(\alpha_{1+mk},\alpha_{2+mk},\dots,\alpha_{k+mk})$ for any $0\leq m\leq d$, in the form of 
$$
(\alpha_{1+mk},\alpha_{2+mk},\dots,\alpha_{k+mk})\begin{pmatrix}
	1 & 1 & \cdots & 1 \\
	1 & \xi^1 & \cdots & \xi^{k-1} \\
	\cdots & \cdots & \cdots & \cdots \\
	1 & \xi^{k-1} & \cdots & \xi^{(k-1)(k-1)}
\end{pmatrix}=0
$$
The coefficient matrix is again in Vandermonde's shape, so there are only $0$ solutions, this means $\overrightarrow{\alpha}=0$. Hence $det(T)\neq 0$.
\end{proof}

	Summing up, $T$ induces a linear map: 
	$$
	T:\tilde{K}^{k(d+1)} \rightarrow \tilde{K}^{u}.
	$$
If $d=Sdeg_A(H)\leq  \frac{u}{k}-1$, i.e. $(d+1)k\leq  u$, then  $T$ is injective (as its matrix consists of linearly independent raws by lemma \ref{L:Tkdu full rank when u=(d+1)k}), this means the linear equation system has  only zero solution, this contradicts to $H\neq 0$.
\end{proof}

\begin{cor}
\label{C:centraliser_free_of_B_j}
Let $C(\partial^k)\subset \hat{D}_1^{sym}\hat{\otimes}_K\tilde{K}$ be the centraliser of $\partial^k$, and $H\in C(\partial^k)$ be a HCP. Then $H$ is totally free of $B_j$.
\end{cor} 

\begin{proof}
The proof follows immediately from lemma \ref{L:homog_comm} and proposition \ref{P:lower estimate}.
\end{proof}

\begin{Prop}
	\label{P:Si}
	Let $Q\in D_1$ be a normalized operator, assume $S\in \hat{D}_{1}^{sym}$ is a Schur operator for $Q$, i.e.  $S^{-1}QS=\partial^q$, such that   $S_{0}=1,S_{-1}=0$. Then we have 
	\begin{enumerate}
		\item $S_{-t}$ is a HCP from $Hcpc(q)$ for any $t\ge 0$ (i.e. it can be written as a HCP in $\hat{D}_1^{sym}\hat{\otimes}_K\tilde{K}$.
		\item If $S_{-t}\neq 0$ then $\frac{t}{q}-1< Sdeg_A(S_{-t})<t$ for any $t> 0$.
		\item $S_{-t}$ is totally free of $B_j$ for any $t\ge 0$.
		\end{enumerate}
\end{Prop} 

\begin{proof}
	The proof is by induction on $t$. All claims are obvious for $S_{-1}$. Suppose they are true from $S_{-1}$ to $S_{-t+1}$. Consider the homogeneous decomposition of $Q$: 
	$$
	Q=Q_q+Q_{q-2}+\ldots =\partial^q+\eta_{0;q-2}\partial^{q-2}+(\eta_{0;q-3}+\eta_{1;q-2}\Gamma_1)\partial^{q-3}+\cdots .
	$$
	Note that $Q_j$ is a HCP for any $j$, $Sdeg_B(Q_j)=-\infty$ for any $j$ and $Sdeg_A(Q_{q-t})\le t-2$ for any $t> 1$. 
	
Since $QS=S\partial^q$, we have the equality of their $(q-t)$ homogeneous components: $(QS)_{q-t}= (S\partial^q)_{q-t}=S_{-t}\partial^q$, i.e. 
\begin{equation}
\label{E:homog}
\partial^q S_{-t}+Q_{q-2} S_{-t+2}+\dots+Q_{q-t}-S_{-t}\partial^q=0.
\end{equation}	
Put $M_t:=-(Q_{q-2} S_{-t+2}+\dots+Q_{q-t})$, so that this equation becomes 
$$[\partial^q,S_{-t}]=M_t.$$ 
Then by lemma \ref{L:HCPC} we get that $M_t$ is a HCP and $Sdeg_A(M_t)\le t-2$ (the only possible term  of degree $t-2$ is $Q_{q-t}$:
 $$
 Q_{q-t}=(\eta_{0;q-t}+\eta_{1;q-t+1}\Gamma_1+\dots+\eta_{t-2;q-2}\Gamma_{t-2})\partial^{q-t},
 $$
whose degree depends on vanishing the coefficient $\eta_{t-2;q-2}$). From formulae \eqref{E:AA}  we get also $Sdeg_B(M_t)=-\infty$. Now applying Lemma \ref{L:ODE sol}, we get $Sdeg_A(S_{-t})\le t-1$ and $Sdeg_B(S_{-t})=-\infty$ (note that since $S_{-t}$ is homogeneous, $S_{-t}$ is a HCP).

In view of formulae \eqref{E:BB}-\eqref{E:AA} we have $Sdeg_B(S_{-t} D^\lambda)=-\infty $  if $\lambda \le 0$. Assume $\lambda >0$. 

For any $p\in\dn$ we have $S^{-1}Q^pS=\partial^{pq}$, and therefore for any $t$ we have equalities similar to \eqref{E:homog}: 	
	$$
	\partial^{pq} S_{-t}+(Q^p)_{pq-2} S_{-t+2}+\dots+ (Q^p)_{pq-t}=S_{-t}\partial^{pq}.
	$$
	Since $Sdeg_B((Q^p)_j)=-\infty$ for all $j$ and $Sdeg_B(S_{-t})=-\infty$ for all $t$, by formulae \eqref{E:AA} we get that the left hand side does not contain $B_j$ (because $\Ord (S_{-t})\le 0$ for any $t$). Therefore, $Sdeg_B(S_{-t}\partial^{pq})=-\infty$ for any $p\in \dn$. Now just note that for any $\lambda >0$ we have $\partial^{\lambda}= \partial^{pq}\int^{pq-\lambda}$ for $p\gg 0$, and therefore $S_{-t}D^{\lambda}=(S_{-t}\partial^{pq})\int^{pq-\lambda }$  does not contain $B_j$ in view of formulae \eqref{E:BB}-\eqref{E:AA}. So, $S_{-t}$ is totally free of $B_j$ for any $t\ge 0$.
	
	The second inequality of item 2 follows from proposition \ref{P:lower estimate}.	    	
\end{proof}

\begin{cor}
	\label{C:Si^(-1) no B}
In the notation of proposition \ref{P:Si} set $\tilde{S}=S^{-1}$. Then we have 
	\begin{enumerate}
	\item $\tilde{S}_0=1$, $\tilde{S}_{-1}=0$.
		\item $\tilde{S}_{-t}$ is a HCP from $Hcpc(q)$ for any $t\ge 0$.
		\item If $S_{-t}\neq 0$ then $\frac{t}{u}-1< Sdeg_A(\tilde{S}_{-t})<t$ for any $t> 0$.
		\item $\tilde{S}_{-t}$ is totally free of $B_j$ for any $t\ge 0$. 
	\end{enumerate}
\end{cor}

\begin{proof}
We can present $S$ as $S=1-S_-$, where $\Ord (S_-)<-1$. Then $S^{-1}=1+\sum_{i=1}^{\infty}S_-^i$, since this series is well defined in the ring $\hat{D}_1^{sym}$. Hence $\tilde{S}_0=1$, $\tilde{S}_1=0$.

Note that any homogeneous component $\tilde{S}_{-t}$, $t>0$ is a finite sum of products of homogeneous components $S_{q}$:
$$
\tilde{S}_{-t}=\sum_{i=0}^{[t/2]}(S_-^i)_{-t}=\sum_{i=1}^{[t/2]}\sum_{q_1+\ldots + q_i=t, q_i> 0}(S_-)_{-q_1}\ldots (S_-)_{-q_i},
$$
and $(S_-)_{j}=\pm S_j$,  thus $\tilde{S}_{-t}$ is a HCP for any $t\ge 0$. Since all $S_q$ are totally free of $B_j$ by proposition \ref{P:Si}, $\tilde{S}_{-t}$ is totally free of $B_j$ for any $t$ by lemma \ref{L:total_B}. By lemmas \ref{L:HCPC} and \ref{P:Si} we have $Sdeg_A((S_-)_{-q_1}\ldots (S_-)_{-q_t})< q_1+\ldots +q_i=t$, and the second inequality of item 3 follows from proposition \ref{P:lower estimate}. 

\end{proof}

\begin{rem}
\label{R:conditon_A_for condition_A}
Note that the statement and proof of proposition \ref{P:Si} remain valid also for regular operators $Q\in \hat{D}_1^{sym}\hat{\otimes}_K\tilde{K}$ whose homogeneous components are HCP from $Hcpc(q)$ totally free of $B_j$ and with $Sdeg_A(Q_{q-i})\le i-2$ for any $i> 1$, cf. theorem \ref{T: P in hat(D) is a dif_op} below.

Note also that  operators $S$, $\tilde{S}$ from these statements are defined over the same field $K$ as the operator $Q$, though their homogeneous components written as HCPs need a formal extension of scalars to be presented in the G-form or standard form. 
\end{rem}

\subsection{Some necessary conditions on normal forms}
\label{S:normal forms}

Let $Q\in D_1$ be a normalized operator as in the previous section. 

\begin{Def}
\label{D:normal_form}
For a given pair of monic operators $Q,P\in D_1$ with $\Ord (Q)=\deg (Q)=q> 0$, $\Ord (P)=\deg (P)=p\ge 0$ we define a {\it normal form} of $P$ with respect to $Q$ as the operator $P':=S^{-1}PS$, where $S$ is a Schur operator for $Q$, i.e. $S^{-1}QS=\partial^q$,  $\Ord (S)=0$, $S_0=1$, $S_{-1}=0$. 
\end{Def} 

\begin{rem}
\label{R:normal_forms}
The Schur operator is not uniquely defined, but up to multiplication by the  invertible elements of order zero from the centralizer  $C(\partial^q)\subset \hat{D}_1^{sym}$ from proposition \ref{T:A.Z 7.1}. By this reason the normal form of the operator $P$  is  not uniquely defined, but up to conjugation by such elements from this centralizer. 

In the same way we can define normal forms for any regular operators $Q,P \in \hat{D}_1^{sym}$. However, by technical reasons we restrict ourself to the case of differential operators (because Schur operators of differential operators satisfy specific properties). 
\end{rem}

In this section we establish several necessary conditions on homogeneous components of a normal form $P'$. 

First note that all homogeneous components of $P$ are HCP, and the homogeneous decomposition of $P$ written in the G-form for all homogeneous components looks like 
$$
P=P_p+P_{p-1}+\ldots =\partial^p+\theta_{0;p-1}\partial^{p-1}+\ldots+(\theta_{0;p-i}+\ldots+\theta_{i-1;p-i}(\Gamma)_{i-1})\partial^{p-i}+\ldots ,
$$
i.e. $Sdeg_A(P_{p-i})< i$ and $Sdeg_B (P_{p-i})=-\infty$ for all $i>0$. Besides, all homogeneous components of $P$ are totally free of $B_j$.  

\begin{Def}
\label{D:conditionA}
We'll say that an operator $P\in  \hat{D}_1^{sym}\hat{\otimes}_K\tilde{K}$ satisfies {\it condition $A_q(k)$}, $q,k\in \dz_+$, $q>1$ if 
\begin{enumerate}
\item
$P_{t}$ is a HCP  from $Hcpc (q)$  for all $t$;
\item
$P_{t}$ is totally free of $B_j$ for all $t$;
\item
$Sdeg_A(P_{\Ord (P)-i})< i+k$ for all $i>0$;
\item
$\sigma (P)$ does not contain $A_{q;i}$, $Sdeg_A(\sigma (P))=k$.
\end{enumerate}
\end{Def}

\begin{ex}
\label{Ex:Schur_op}
From previous section we know that $S, S^{-1}$ satisfy condition $A_q(0)$, where $S$ is any monic Schur operator for normalised $Q\in D_1$. Moreover, if $S$ is a monic operator from $\hat{D}_1^{sym}$ of order zero satisfying condition $A_q(0)$ (but not necessary with $S_{-1}=0$), then $S^{-1}$ also satisfies the same properties. The proof is just the same as in corollary \ref{C:Si^(-1) no B}. 
\end{ex}

With the help of this definition we can prove the following criterion.  

\begin{theorem}
\label{T: P in hat(D) is a dif_op}
	The following statements are equivalent:
	\begin{enumerate}
		\item $P\in \hat{D}_1^{sym}\hat{\otimes}_K\tilde{K}$ is a {\it differential operator} (i.e. $P\in D_1\otimes_K\tilde{K}$)  with constant highest symbol.
		\item $\forall p>1$, $P\in \hat{D}_1^{sym}\hat{\otimes}_K\tilde{K}$ satisfies condition $A_p(0)$ with an extra property: all homogeneous components $P_j$ don't contain  $A_i$.
		\item $\exists p>1$, $P\in \hat{D}_1^{sym}\hat{\otimes}_K\tilde{K}$ satisfies condition $A_p(0)$ with an extra property: all homogeneous components $P_j$ don't contain  $A_i$.
	\end{enumerate}
\end{theorem}

\begin{proof}
To simplify notations, we'll assume in the course of proof that $D_1$ is defined over $\tilde{K}$ (i.e. we assume $K=\tilde{K}$). 

	1$\Rightarrow$2: It is easy to see that any differential operator $P\in D_1$ with constant highest symbol, i.e. $P=a_q\partial^q+\sum_{i=1}^q a_{q-i}\partial^{q-i}$ with $a_q\in K$, satisfies condition $A_p(0)$ for any $p>0$ with an extra property: all homogeneous components $P_i$ don't contain  $A_i$, cf. lemma \ref{L:correct_HPC}.

	2$\Rightarrow$3 is obvious.
	
	3$\Rightarrow$1: We need to show $P\in D_1$. First let's show $\forall r\in\dz$  $P_r\in D_1$.  This is obvious  when  $r\ge 0$, because $P_r$ don't contain neither $B_j$ nor $A_i$. 
	
	In the case when $r<0$ the proof is by induction on $r$. Consider first $P_{-1}$; let's write it in the G-form. If $P_{-1}\neq 0$, according to the assumptions, suppose 
	$$
	P_{-1}=\sum_{0\leq m\leq d_{-1}}p_{m,0;-1}\Gamma_{m}D^{-1}.
	$$
	Since it's totally free of $B_j$, there won't be $B_j$ in $P_{-1} D$. Since
	$$
	P_{-1} D=-p_{0,0;-1}B_1+\sum_{0\leq m\leq d_{-1}}p_{m,0;-1}\Gamma_{m}
	$$
	according to Lemma \ref{L:HCPC} (see the calculations in the proof of item 3), hence by the uniqueness of HCPC (c.f. Lemma \ref{L:correct_HPC}) we know  $p_{0,0;-1}=0$. Thus 
	\begin{multline*}
	P_{-1}=\sum_{1\leq m\leq d_{-1}}p_{m,0;-1}\Gamma_{m}D^{-1}
	=\sum_{1\leq m\leq d_{-1}}p_{m,0;-1}(x\partial)^mD^{-1}\in D_1=\sum_{1\leq m\leq d_{-1}}p_{m,0;-1}(x\partial)^{m-1}x\in D_1
    \end{multline*}
	For $r>1$, consider $P_{-r}=\sum_{0\leq m\leq d_{-r}}p_{m,0;-r}\Gamma_mD^{-r}$. By the same reason $P_{-r}D^r$ doesn't contain $B_1$, so we get $p_{0,0,-r}=0$. Hence 
	\begin{multline*}
	P_{-r}=\sum_{1\leq m\leq d_{-r}}p_{m,0;-r}\Gamma_{m}D^{-r}=\sum_{1\leq m\leq d_{-r}}p_{m,0;-r} \Gamma_{m-1}xD^{-r+1}
	\\=\sum_{1\leq m\leq d_{-r}}p_{m,0;-r} x(\Gamma_1+1)^{m-1}D^{-r+1}=x\sum_{0\leq n\leq d_{-r}-1}\tilde{p}_{n,0;-r} \Gamma_n D^{-r+1}
    \end{multline*}
	where 
	$$
	\tilde{p}_{n,0;-r}=\sum_{m=1}^{d_{-r}}\binom{m-1}{n}p_{m,0;-r}
	$$
	
	Denote $H=\sum_{0\leq n\leq d_{-r}-1}\tilde{p}_{n,0;-r}\Gamma_{n}D^{-r+1}$, then $P_{-r}=xH$, it is written in G-form and totally free of $B_j$. Then by Lemma \ref{L:H tfree of Bj iff xH tfree of Bj} (see below) $H$ is also totally free of $B_j$. By induction, $H\in D_1\Rightarrow P_{-r}=xH\in D_1$. 
	
	Since $P$ satisfies $A_p(0)$, we have for all $i>0$ $Sdeg_A(P_{\Ord(P)-i})<i\Rightarrow deg(P_{\Ord(P)-i})<\Ord(P)$. 	Also we have $ P_{\Ord(P)}=\sigma(P)\in D_1$ with $deg(\sigma(P))=\Ord(\sigma(P))$. Thus 
$$P=\sum_{i=0}^{\infty} P_{\Ord(P)-i}\in D_1,$$ 
	$deg(P)=\Ord(P)$ and $P$ has a highest constant symbol.
\end{proof}

\begin{lemma}
	\label{L:H tfree of Bj iff xH tfree of Bj}
	Suppose $H$ is a HCP from $Hcpc(p)$ and doesn't contain $A_i$, with $Sdeg_A(H)<\infty$. Then $H$ is totally free of $B_j$ if and only if $xH$ is totally free of $B_j$.
\end{lemma}

\begin{proof}
	Suppose $H$ is totally free of $B_j$, thus $H$ should be in the form of 
	$$
	H=\sum_{0\leq m\leq d} h_{m,0;r}\Gamma_m D^r
	$$
	we want to show $xH$ is totally free of $B_j$, i.e to show for any $k\in \mathbb{Z}$, $xHD^k$ doesn't contain $B_j$. Since $H$ is totally free of $B_j$, we know $Sdeg_B(HD^k)=-\infty$. According to Lemma \ref{L:HCPC} (notice $-\Ord(\int)=1$), we know 
	$$
	Sdeg_B(\int H D^k)\leq 1
	$$
	Hence if we write $\int H D^k$ into G-form, it should be like 
	$$
	\int HD^k=(\cdots)D^{r+k-1}+\lambda B_1D^{r+k-1}
	$$
	Where $\cdots$ means the polynomial of $\Gamma_1=(x\partial)$ and $\lambda\in \tilde{K}$. But notice that $\partial B_1=0$, so the terms with $B_1$ will be eliminate in $xHD^k$, i.e. 
	$$
	xH D^k=(x\partial)(\int HD^k)=(\cdots)D^{r+k-1}
	$$
	Hence $xH$ is totally free of $B_j$.
	
	On the other hand, if we know $xH$ is totally free of $B_j$, suppose $H$ is not totally free of $B_j$, hence there exist $k$, such that $HD^k$ contains $B_j$, suppose $Sdeg_B(HD^k)=j>0$, with
	$$
	HD^k=\sum_{m=0}^{d_r}\tilde{h}_{m,0;r}\Gamma_mD^{r+k}+\sum_{t=1}^{j-1}g_{t;r+k}B_tD^{r+k}+\lambda B_jD^{r+k}
	$$
	Since 
	$$
	xB_jD^{r+k}=\Gamma_1\int B_jD^{r+k}=\Gamma_1 B_{j+1}\int D^{r+k}=jB_{j+1}\int D^{r+k}
	$$
	Notice that $\int D^{r+k}=D^{r+k-1}$ when $r+k\leq 0$ and $(1-B_1)D^{r+k-1}$ when $r+k>0$, but $j>1$ and we know $B_{j+1}B_1=\delta_1^{j+1}B_1=0$, hence  $B_{j+1}\int D^{r+k}=B_{j+1}D^{r+k-1}$, thus
	$$
	xHD^k=\sum_{m=0}^{d_r+1}\tilde{\tilde{h}}_{m,0;r}\Gamma_mD^{r+k-1}+\sum_{t=1}^{j}\tilde{g}_{t;r+k}B_tD^{r+k-1}+j\lambda B_{j+1}D^{r+k-1}
	$$
	This is a contradiction with $xH$ totally free of $B_j$.
	\end{proof}

\begin{rem}
\label{R:criterion}
Note that the criterion \ref{T: P in hat(D) is a dif_op} holds also if $P$ is defined over $K$. In this case in items 2 and 3 we need to add that $P_j$ are defined over $K$. The proof remains the same.
\end{rem}

\begin{lemma}
\label{L:conditionA}
Suppose $P,Q\in \hat{D}_1^{sym}\hat{\otimes}_K\tilde{K}$ satisfy conditions $A_q(k_1)$, $A_q(k_2)$ correspondingly. Then 
$PQ$ satisfies condition $A_q(k_1+k_2)$.
\end{lemma}

\begin{proof}
Let $p=\Ord (P)$, $q=\Ord (Q)$. Since $\sigma (P)$, $\sigma (Q)$ does not contain $A_{q;i}$ and $B_j$, they are differential operators, i.e. belong to $D_1$, and therefore $0\neq \sigma (P)\sigma (Q)=\sigma (PQ)$, cf. remark \ref{R:ord_properties}, and thus $\sigma (PQ)$ does not contain $A_{q;i}$ and $B_j$ and $\Ord (PQ)=p+q$. Moreover, in this case $Sdeg_A(\sigma (PQ))= Sdeg_A(\sigma (P))+Sdeg_A(\sigma (Q))=k_1+k_2$, cf. formulae \eqref{E:AA}. For other homogeneous components of $PQ$ we have
$$
(PQ)_{p+q-i}=\sum_{i_1+i_2=i}P_{p-i_1}Q_{q-i_2},
$$
is a HCP from $Hcpc(q)$ and  
$$
Sdeg_A(PQ)_{p+q-i}\le \max\{Sdeg_A(P_{p-i_1}Q_{q-i_2})\}< i_1+k_1+i_2+k_2=i+k_1+k_2
$$
for all $i>0$ by lemmas \ref{L:obvious}, \ref{L:HCPC}. Besides, $(PQ)_{p+q-i}$ is totally free of $B_j$ by lemma \ref{L:total_B} for all $i>0$. 
\end{proof}

\begin{cor}
	\label{C:P'}
	Suppose $Q\in D_1$ is a normalized operator with  $\Ord (Q)=\deg (Q)=q> 0$. 
	Suppose $P\in \hat{D}_1^{sym}\hat{\otimes}_K\tilde{K}$  satisfies condition $A_q(0)$, $\Ord (P)=p$. 
	 Put $P'=S^{-1}PS$, where $S$ is a Schur operator for $Q$ from proposition \ref{P:Si}. 
	 
	 Then $P'$ satisfies condition $A_q(0)$.
\end{cor}

\begin{proof}
By proposition \ref{P:Si} and corollaries \ref{C:HCPC(k)}, \ref{C:Si^(-1) no B} the operators $S, S^{-1}$ satisfy condition $A_q(0)$. So, our claim immediately follows from lemma \ref{L:conditionA}. 

\end{proof}

\begin{rem}
Note that if $P$ is defined over $K$, then $P'$ will be defined over $K$ too. 
\end{rem}

\section{Normal forms for commuting operators}
\label{S:normal_forms_for_commuting}

Let $Q\in D_1$ be a normalized operator as in section \ref{S:necessary conditions on the Schur}. Let $P\in D_1$ be a monic operator of positive order $p$ and $[P,Q]=0$. Fix $k=q=\Ord(Q)$, suppose $\tilde{K}$ is an algebraic closure of $K$, $A_i:=A_{k;i}$ as above. 

The famous Burchnall-Chaundy lemma (\cite{BC1}) says  that any two commuting differential operators $P,Q\in D_1:=K[[x]][\partial ]$ are algebraically dependent. More precisely, if the orders $p,q$ of operators $P,Q$ are coprime\footnote{i.e. the rank of the ring $K[P,Q]$ is 1, see e.g. \cite{Zheglov_book} for relevant definitions, in particular \cite[Lemma 5.23]{Zheglov_book} for a proof of the Burchnall-Chaundy lemma in general form. The statement about the form of polynomial follows easily from the proof.}, then there exists an irreducible polynomial $f(X,Y)$ of weighted degree $v_{p,q}(f)=pq$ of {\it special form} (here the weighted degree is defined as in Dixmier's paper \cite{Dixmier}, cf. item 4 in the List of Notations below): $f(X,Y) = \alpha X^q\pm Y^p+\ldots $ (here $\ldots$ mean terms of lower weighted degree, $0\neq\alpha\in K$; in particular, for coprime $p$ and $q$ the polynomial $f$ is automatically irreducible), such that $f(P,Q)=0$. A similar result for commuting operators of rank $r$ was established in \cite{Wilson} (cf. \cite[Th. 2.11]{Previato2019}), in this case $q=\ord (Q)/r$, $p=\ord (P)/r$, and again $GCD (p,q)=1$.

In this section we give a convenient description of the centraliser $C(\partial^q)\subset \hat{D}_1^{sym}\hat{\otimes}\tilde{K}$ and of normal forms of the operator $P$ with respect to $Q$. With the help of this description we give a new parametrisation of torsion free sheaves of rank one with vanishing cohomology groups on a projective  curve (according to the well known classification theory of commuting ordinary differential operators such sheaves describe rank one subrings of commuting operators with a given spectral curve, cf. \cite[Th. 10.26]{Zheglov_book}).

Consider the subring $\tilde{K}[A_1,\ldots ,A_{k-1}] \subset \hat{D}_1^{sym}\hat{\otimes}\tilde{K}$. Clearly, $\tilde{K}[A_1,\ldots, A_{k-1}]\cong \tilde{K}[x]/(x^k-1)$. Note that we have an isomorphism of $\tilde{K}$-algebras 
$$
\Phi : \tilde{K}[A_1,\ldots ,A_{k-1}]\rightarrow \tilde{K}^{\oplus k} , \quad P\mapsto (1\circ P, \ldots , (\partial^{k-1}\circ P){\partial^{-k+1}})
$$
(here $\partial^l\circ P$ is the notation from definition \ref{D:regular}, i.e. if $P=\sum_{i=0}^{k-1}p_iA_i$, then $\partial^l\circ P=(\sum p_i\xi^{il}){\partial^{l}}$, and $\tilde{K}^{\oplus k}$ denotes the semisimple algebra - the direct sum of algebras $\tilde{K}$). Indeed, $\Phi$ is obviously linear, and, since  
$$
\Phi(h_{l}A_1^l\cdot m_{t}A_1^t)=(h_lm_t,h_lm_t\xi^{l+t},\ldots,h_lm_t\xi^{(l+t)(k-1)})=\Phi(h_{l}A_1^l)\cdot\Phi(m_{t}A_1^t),
$$
$\Phi$ is a $\tilde{K}$-algebra homomorphism. It's easy to see that it is surjective and injective. 

\begin{rem}
\label{R:Phi_over_K}
Note that, by definition of $\Phi$, if $P\in \tilde{K}[A_1]$ belong to $\hat{D}_1^{sym}$ (i.e. all coefficients of the series representing $P$ in  $\hat{D}_1^{sym}$ belong to $K$), then $\Phi (P)\in K^{\oplus k}$, i.e. $\Phi$ is compatible with the extension of scalars of $\hat{D}_1^{sym}$.
\end{rem}

Now consider the skew polynomial ring $\crr = \tilde{K}^{\oplus k} [D,\sigma ]$, where \\
$\sigma (a_0,\ldots ,a_{k-1})= (a_{k-1}, a_0, \ldots , a_{k-2})$ \footnote{We use a standard notations and constructions from the books \cite{Cohn} and \cite{NNR}. A short self contained exposition of all necessary constructions and facts see e.g. in \cite{Zheglov_book}, Ch. 2,3 and 13.1. 

Recall that this notation means that we have the following commutation relation between $D$ and $(a_0, \ldots ,a_{k-1})$: $(a_0, \ldots ,a_{k-1})D=D\sigma (a_0, \ldots ,a_{k-1})$. }. It is a right and left noetherian ring. The multiplicatively closed subset $S=\{ D^k, k\ge 0\}$ obviously satisfies the right Ore condition, consists of regular elements, and ${\bf ass} (S)=0$. So, the right quotient ring $\mathfrak{B}=\crr_S$ exists. Clearly, 
$$
\mathfrak{B}\simeq\tilde{K}^{\oplus k}[D,D^{-1}]=\{\sum_{l=M}^NP_lD^l | \quad P_l\in \tilde{K}^{\oplus k}\}\simeq \tilde{K}[A_1][D,D^{-1}],
$$
i.e. any element from $\mathfrak{B}$ can be written as a Laurent polynomial, and the commutativity relations of polynomials are given above: $D^{-1}a=\sigma (a) D^{-1}$, $a\in \tilde{K}^{\oplus k}$.

Let $C(\mathfrak{B})$ be the center of $\mathfrak{B}$. Obviously, $D^k,D^{-k}\in C(\mathfrak{B})$. Since an element $(h_0,\ldots,h_{k-1})\in \tilde{K}^{\oplus k}$ doesn't commute with any $D^l$, if not all of $h_i$ equal to each other,  we have 
$$
C(\mathfrak{B})\cong \tilde{K}[D^k,D^{-k}],
$$
where $\tilde{K}$ is diagonally embedded into $\tilde{K}^{\oplus k}$. So, $\mathfrak{B}$ is a finite dimensional algebra over its center.

\begin{lemma}
\label{L:B is matrix algebra over C(B)}
	There is an isomorphism of $\tilde{K}$-algebras
	$$
	\mathfrak{B}\cong M_k(C(\mathfrak{B})).
	$$
\end{lemma} 

\begin{proof}
	Consider $\psi:\mathfrak{B}\rightarrow M_k(C(\mathfrak{B}))$, with 
	$$
	\psi\begin{pmatrix}
		h_0 \\
		h_1 \\
		\cdots \\
		h_{k-1}
	\end{pmatrix}=\begin{pmatrix}
	h_0 &  &  &  \\
	& h_1 &  &  \\
	&  & \cdots &  \\
	&  &  & h_{k-1}
	\end{pmatrix}\quad \psi(D)=T:=\begin{pmatrix}
	& 1 &  & \cdots &  \\
	&  & 1 & \cdots &  \\
	\cdots & \cdots & \cdots & \cdots & \cdots \\
	&  &  & \cdots & 1 \\
	D^k &  &  & \cdots & 
	\end{pmatrix}
	$$
	with $\psi(D^l)=T^l$, and extend $\psi$ by linearity. Direct calculations show that $\psi$ is a homomorphism of $\tilde{K}$-algebras. Now consider 
	$$
	H_{ij}=\begin{cases}
		(0,\cdots,1,\cdots,0)D^{j-i}&i\leq j
		\\(0,\cdots,1,\cdots,0)D^{j+k-i}&i>j
	\end{cases}
	$$  
	where $1$ is located at the $i$th entry, so that $\psi(H_{ij})=E_{ij}D^k$, $i>j$ or $E_{ij}$, $i\le j$. This means $\psi$ is a surjective. Obviously, $\mathfrak{B}$ has  dimension $k^2$ over $C(\mathfrak{B})$ and $\dim_{C(\mathfrak{B})}(M_k(C(\mathfrak{B})))=k^2$, too. Besides, $\psi (C(\mathfrak{B}))=C(M_k(C(\mathfrak{B})))= C(\mathfrak{B})\cdot Id$. So  $\psi$ is an isomorphism of $\tilde{K}$-algebras.
\end{proof}

Now consider the ring of skew pseudo-differential operators 
$$
E_k:=\tilde{K}[\Gamma_1, A_1]((\tilde{D}^{-1}))=\{\sum_{l=M}^{\infty}P_l\tilde{D}^{-l} | \quad P_l\in \tilde{K}[\Gamma_1, A_1]\} \simeq \tilde{K}^{\oplus k}[\Gamma_1]((\tilde{D}^{-1}))
$$
with the commutation relation as above (here $\tilde{K}[\Gamma_1, A_1]$ is a commutative subring in $\hat{D}_1^{sym}$):\footnote{The ring $E_k$ is constructed in the same way as splittable  local skew fields, cf. \cite{Zheglov_izv}}
$$
\tilde{D}^{-1}a=\sigma (a)\tilde{D}^{-1}, \quad a\in \tilde{K}[\Gamma_1, A_1] \quad \mbox{where \quad }
\sigma (A_1)= \xi^{-1} A_1, \quad \sigma (\Gamma_1)=\Gamma_1+1.
$$
The ring $E_k$ is endowed with a natural discrete pseudo-valuation, which we will denote as $-\ord_{\tilde{D}}$ (i.e. $\ord_{\tilde{D}}(\sum_{l=M}^{\infty}P_l\tilde{D}^{-l})= M$).  We extend the usual terminology used in this paper also for operators from $E_q$ (such as the notion of the highest coefficient, monic operators, etc.)

Denote $\widehat{Hcpc}_B(k)$ as the $\tilde{K}$-subalgebra in $\hat{D}_1^{sym}\hat{\otimes}\tilde{K}$ consisting of operators whose homogeneous components are HCPs totally free of $B_j$ (cf. lemma \ref{L:total_B}). 

\begin{lemma}
\label{L:embedding_Phi}
The map 
$$
\hat{\Phi}: \widehat{Hcpc}_B(k) \longrightarrow E_k,
$$
defined on monomial HCPs from $\widehat{Hcpc}_B(k)$ as $\hat{\Phi} (a A_j\Gamma_iD^l):=a \Phi (A_j)\Gamma_i\tilde{D}^l$ and extended by linearity on the whole $\tilde{K}$-algebra $\widehat{Hcpc}_B(k)$, is an embedding of $\tilde{K}$-algebras.
\end{lemma}

\begin{proof}
Since all operators in $\widehat{Hcpc}_B(k)$ are totally free of $B_j$, the proof is almost obvious in view of formulae \eqref{E:BB}-\eqref{E:AA}.
\end{proof}

\begin{rem}
\label{R:hatPhi_over_K}
Again as in remark \ref{R:Phi_over_K}, if $P\in \hat{D}_1^{sym}\cap \widehat{Hcpc}_B(k)$, then $\hat{\Phi} (P)$ will be an operator with coefficients from $K$. 
\end{rem}

\begin{rem}
\label{R:algebraic_dependence_of_normal_forms}
Note that the ring $\mathfrak{B}$ is naturally embedded into the ring $E_k$. In particular, all normal forms of a monic differential operator $P$ with respect to the commuting with $P$ differential operator $Q$, or more generally any operator from the centralizer $C(\partial^k)$, can be embedded in $E_k$ via $\hat{\Phi}$. Recall that any such normal form is defined up to conjugation with an operator from the centralizer $C(\partial^k)$, and $C(\partial^k)\subset \widehat{Hcpc}_B(k)$ by corollary \ref{C:centraliser_free_of_B_j}.  

Next, note that any operator from the centralizer $C(\partial^k)$ embedded into $\mathfrak{B}$ goes under the isomorphism $\psi : \mathfrak{B}\cong M_k(C(\mathfrak{B}))$ to a matrix with entries belonging to the subring $\tilde{K}[D^k]$. Indeed,  we know from lemma \ref{L:homog_comm} that all coefficients of homogeneous terms of negative order satisfy the following conditions: if $P\in \hat{\Phi} (C(\partial^k))\subset \mathfrak{B}$, and $p_l\in \tilde{K}^{\oplus k}$ denote its coefficients, then always $\tilde{p}_{l,j}=0$ for $j=0,\ldots , -l-1$ if $l<0$. Therefore all homogeneous terms of negative order will go to matrices with {\it constant coefficients} (see the proof of lemma \ref{L:B is matrix algebra over C(B)}), and terms of non-negative order go to matrices with entries belonging to the subring $\tilde{K}[D^k]$. Finally, we can observe that in fact we get an isomorphism $C(\partial^k)\simeq M_k(\tilde{K}[D^k])$.

Therefore, for any $P\in C(\partial^k)$ the characteristic polynomial $\det (\psi (\hat{\Phi}(P))-\lambda)\in \tilde{K}[\lambda , D^k]$ defines an algebraic relation between $P$ and $\partial^k$ or, in degenerate cases, it defines an algebraic dependence over $\tilde{K}$. Note that if $P'$ is a normal form  of a monic differential operator $P$ with respect to the commuting with $P$ differential operator $Q$, and the order of $P$ is coprime with the order of $Q$, then $P'$ is monic and $\det (\psi (\hat{\Phi}(P))-\lambda)\in \tilde{K}[\lambda , D^k]$ is a polynomial of the form similar to the Burchnall-Chaundy polynomial, i.e. $\det (\psi (P)-\lambda) =  \pm\lambda^q\pm D^{kp}+\ldots $, in particular it is irreducible and therefore coincides with the Burchnall-Chaundy polynomial for $P$ and $Q$ up to a multiplicative constant. 
\end{rem}

\begin{rem}
\label{R:vector_form}
Elements of the centralizer $C(\partial^k)$ embedded to $\mathfrak{B}$ we'll call as a {\it vector form} presentation, and elements embedded to $M_k(\tilde{K}[D^k])$ -- as a {\it matrix form} presentation. 

Translating the description of the centralizer $C(\partial^k)$ given in proposition \ref{T:A.Z 7.1}, item 2, into the vector form, we get that $\hat{\Phi}(C(\partial^k))$ consists of Laurent polynomials in 
$\tilde{D}$ with coefficients from $K^{\oplus k}$ and with additional conditions: the coefficient $s_i$ at $\tilde{D}^{-i}$, $i>0$, has a shape $s_{i,j}=0$ for $j=0, \ldots , i-1$. 
\end{rem}

Embedding normal forms to the ring $E_k$, we can transform them to a simpler form by conjugation:
\begin{lemma}
\label{L:conjugation in E_q}
Assume $\tilde{P}\in \tilde{K}^{\oplus k}((\tilde{D}^{-1}))\subset E_k$ is a monic operator with $\ord_{\tilde{D}}(\tilde{P})=p$. Let $p=dn$, $k=q=dm$, where $d=GCD (p,q)$. 

Then there exists a monic invertible operator $S\in \tilde{K}^{\oplus k}((\tilde{D}^{-1}))\subset E_k$ with $\ord_{\tilde{D}}(S)=0$ such that all coefficients of the operator $S^{-1}\tilde{P}S$ commute with $\tilde{D}^d$.
\end{lemma}

\begin{proof}
We will find the operator $S$ as the limit of a Cauchy sequence\footnote{with respect to the topology defined by the pseudo-valuation $-\ord_{\tilde{D}}$}. Assume 
$$
\tilde{P}=\tilde{D}^{dn}+\sum_{i=1}^{\infty}p_i\tilde{D}^{dn-i} \quad p_i\in \tilde{K}^{\oplus k}, 
$$
and let $p_l=(p_{l,0}, \ldots ,p_{l,q-1})$ be the first coefficient not commuting with $\tilde{D}^d$. Consider the operator $S_l:=1+s_l \tilde{D}^{-l}$, where $s_l=(s_{l,0},\ldots , s_{l,q-1})$. Then an easy direct calculation shows that
$$
S_l^{-1}\tilde{P} S_l=\tilde{D}^{dn}+\ldots + (p_l-s_l+\sigma^{p}(s_l))\tilde{D}^{dn-l}+\mbox{terms of lower order},
$$
where $\ldots$ are the same terms of $\tilde{P}$. Consider the following $d$ systems of linear equations: for $i=0, \ldots , d-1$ [mod $d$] set $b_i:=\sum_{k=0}^{m-1}p_{l,[(i+pk) \mbox{mod $q$}] }$, then the $i$-th system is (below all indices are considered modulo $q$)
$$
\begin{cases}
p_{l,i}-s_{l,i}+s_{l,i+p}=b_i/m \\
p_{l,i+p}-s_{l,i+p}+s_{l,i+2p}=b_i/m \\
\ldots \\
p_{l,i+(m-1)p}-s_{l,i+(m-1)p}+s_{l,i}=b_i/m
\end{cases}
$$ 
This system is solvable and give explicit formulae for unknown variables $s_{l,j}$: for $r=1,\ldots ,(m-1)$ from the first $(m-1)$ equations we get 
$$s_{l,i+rp}=r b_i/m+s_{l,i}-(\sum_{j=0}^{r-1} p_{l,i+jp})$$
($s_{l,i}$ is a free parameter), and the last equation becomes an identity under substitution of these values. Taking such $s_l$ we get
$$
\tilde{p}_l:= p_l-s_l+\sigma^{p}(s_l)= (b_0/m, \ldots , b_{d-1}/m, b_0/m, \ldots , b_{d-1}/m, \ldots )
$$ 
which obviously commutes with $\tilde{D}^d$. 

It's easy to see that the system $\{\prod_{i=1}^j S_i\}$, $j\ge 1$ is a Cauchy system in $E_q$, and its limit $S$ is the needed operator. 
\end{proof}

\begin{cor}
\label{C:Schur_subring}
Let $B'\subset \tilde{K}^{\oplus k}((\tilde{D}^{-1}))\subset E_q$ be a commutative subring containing $\tilde{D}^q$ and a monic operator $P'$ of order $p$. Suppose $GCD(p,q)=1$. 

Then $S^{-1}B'S\subset \tilde{K}((\tilde{D}^{-1}))$ for an operator $S$ from lemma \ref{L:conjugation in E_q}.
\end{cor} 

The proof is obvious. 

\begin{rem}
\label{R:S_over_K}
Clearly, if $\tilde{P}\in K^{\oplus k}((\tilde{D}^{-1}))$, then the operator $S$ from lemma \ref{L:conjugation in E_q} also belongs to $K^{\oplus k}((\tilde{D}^{-1}))$, cf. remark \ref{R:hatPhi_over_K}.
\end{rem}

Normal forms of commuting differential operators can be normalised (s.t. the normalised normal form will be uniquely defined in its conjugacy class). If the orders of these operators are coprime, such a normalisation can be easily presented:

\begin{lemma}
\label{L:normalisation}
Let $Q,P$ be differential operators as in the beginning of this section. Assume $GCD(p,q)=1$. 
For $i=1, \ldots , q-1$ define the sets $N_i:=\{n\in \dn |\quad \mbox{($np$ mod $q$)$\ge i$}\}$. 
  
Then there is a normalised normal form $P'$ of $P$ with respect to $Q$ uniquely defined in the conjugacy class  by elements from the centraliser $C(\partial^q)$. Namely, we can explicitly describe its image under the embedding $\hat{\Phi}$ as follows:
$$
\hat{\Phi}(P')=\tilde{D}^p +\sum_{l=-q+1}^{p-1} p_l \tilde{D}^l, \quad p_l=(p_{l,0},\ldots ,p_{l,q-1})\in \tilde{K}^{\oplus q},
$$
where $p_{l,j}=0$ for indices $j$ satisfying the following rules (all indices are taken mod $q$):
$$
\begin{cases}
\mbox{If $p\ge q$ then} 
\begin{cases}
p_{l,(j-1)p}=\mbox{$0$ for $j\in N_{p-l}$}\quad l>p-q \\
p_{l,j}=\mbox{$0$ for $j=0,\ldots ,-l-1$} \quad l<0
\end{cases}
\\
\mbox{If $p< q$ then}
\begin{cases}
p_{l,(j-1)p}=\mbox{$0$ for $j\in N_{p-l}$}\quad l\ge 0 \\
p_{l,j}=\mbox{$0$ for $j=0,\ldots ,-l-1$ or if $j=(k-1)p$ and $k\in N_{p-l}$}\quad p-q<l<0 \\
p_{l,j}=\mbox{$0$ for $j=0,\ldots ,-l-1$} \quad l\le p-q
\end{cases}
\end{cases}
$$
Let's call coefficients $p_{l,j}$ supplementary to the list above as {\it coordinates} of  $P'$. If $P'$ and $P''$ are two operators with different values of coordinates, then there are no invertible operators $S\in C(\partial^q)$ such that $P'=S^{-1}P''S$, and there are no invertible $1\neq S\in C(\partial^p)$ such that $P'=S^{-1}P'S$. 

\end{lemma}

\begin{proof}
To show that such normalised form exist, we can follow the arguments in lemma \ref{L:conjugation in E_q}. If $P'$ is any given normal form of $P$ with respect to $Q$,  we will find a monic operator of order zero $S\in C(\partial^q)$ (thus it will be automatically invertible) step by step, as a product of operators $S_l=1+s_l \int^{l}\in C(\partial^q)$, $l=1,\ldots ,q-1$. Since $\hat{\Phi}$ is an embedding of rings, we can provide all calculations in the ring $E_q$. Note that by definition of $\hat{\Phi}$ and from lemma \ref{L:homog_comm} we get $\hat{\Phi}(S_l)=1+\tilde{s}_l \tilde{D}^{-l}$, where $\tilde{s}_{l,j}=0$ for $j=0, \ldots , l-1$. Using calculations from lemma \ref{L:conjugation in E_q}, we get for 
$\tilde{P}'':=\hat{\Phi}( S_l^{-1}P'S_l)= \tilde{D}^p+\ldots + \tilde{p}_l\tilde{D}^{p-l}+\ldots$ that
$$
\begin{cases}
p_{l,0}-\tilde{s}_{l,0}+\tilde{s}_{l,p}=\tilde{p}_{l,0} \\
p_{l,p}-\tilde{s}_{l,p}+\tilde{s}_{l,2p}=\tilde{p}_{l,p} \\
\ldots \\
p_{l,(q-1)p}-\tilde{s}_{l,(m-1)p}+\tilde{s}_{l,0}=\tilde{p}_{l,(q-1)p}
\end{cases}
$$ 
because $GCD(p,q)=1$ (here again all indices are taken modulo $q$). Now we have $\tilde{s}_{l,0}=0$ for all $l=1,\ldots ,q-1$, and, starting with the first equation, we can see that for $j$-th equation, where $j\in N_l$, we can find $\tilde{s}_{l,jp}$ such that $\tilde{p}_{l,(j-1)p}=0$. On the other hand, $\tilde{s}_{l,jp}=0$ for all $j\notin N_l$. Thus, we uniquely determine $\tilde{s}_l$ such that conditions of lemma satisfied for $(p-l)$-th homogeneous component of $\tilde{P}''$. 

On the other hand, since $\tilde{P}''\in C(\partial^q)$, we know from lemma \ref{L:homog_comm} that always $\tilde{p}_{l,j}=0$ for $j=0,\ldots , -l-1$ if $l<0$ (in particular, the number of zero coefficients for $p-q<l<0$ in case $p<q$ is constant). Taking $S=\prod_{j=1}^{q-1}S_j$ we get the needed operator: the normal form $S^{-1}P'S$ will satisfy all conditions as stated. This completes the proof of the first statement. 

To prove the second statement (the uniqueness of such normalised normal form up to conjugation), first note that if $S\in C(\partial^q)$ is invertible and $\Ord (S)=t>0$, then the coefficient of the HCP $S_t$ must be a zero divisor. Indeed, assume the converse. Suppose $\Ord (S^{-1})=r$. Then $(SS^{-1})_{t+r}\neq 0$ hence $r=-t<0$. From lemma \ref{L:homog_comm} we know then the coefficient of the HCP $(S^{-1})_r$ is a zero divisor, hence the coefficient of the HCP  $(S^{-1}S)_{t+r}=1$ must be a zero divisor, a contradiction. By the same reason there are no invertible $S$ with $\Ord (S)<0$ (because all its homogeneous components will be zero divisors by lemma \ref{L:homog_comm}). 

Now suppose we have two normalised normal forms $P'$ and $P''$ with different values of coordinates, and 
$P'=S^{-1}P''S$ for some $S\in C(\partial^q)$. Suppose $\Ord (S)=t\ge 0$. From the equality we have $S_t\partial^p=\partial^pS_t$. Since $p$ and $q$ are coprime, the coefficient of $S_t$ must be a constant. Thus, it is not a zero divisor and therefore $t=0$. Without loss of generality we can assume  $S_0=1$. But then the same equations as above show that all other homogeneous components of $S$ must be zero (as all coefficients of  $P'$ and $P''$ from the list in the formulation are zero) and therefore the equality $P'=S^{-1}P''S$ is impossible. 
\end{proof}

Now we are ready to give a description of a new parametrisation of torsion free sheaves of rank one with vanishing cohomology groups on a projective  curve. First let's recall the classification theorem of rank one commutative subrings of differential operators. 

\begin{theorem}{(\cite[Th. 10.26]{Zheglov_book})}
\label{T:classif2}
There is a one-to-one correspondence 
$$
[B\subset {D_1}\mbox{\quad of rank $1$}]/\thicksim  \longleftrightarrow  [(C,p,\cf  )\mbox{\quad of rank $1$}]/\simeq 
$$
where 
\begin{itemize}
\item
$[B]$ means a class of equivalent commutative elliptic subrings (i.e. $B$ containing a monic differential operator),  where   $B\cong B'$ iff $B=f^{-1}B'f$, $f\in D_1^*=K[[x]]^*$. 
\item
$\thicksim$ means "up to a scale automorphism $x\mapsto c^{-1}x$, $\partial\mapsto c\partial$".
\item $C$ is a (irreducible and reduced) projective curve over $K$, $p$ is a regular $K$-point and $\cf$ is a torsion free sheaf of rank one with $H^0(C, \cf)=H^1(C,\cf )=0$ (a spectral sheaf). 
\item  $[]/\simeq$ on the right hand side means an equivalence class of triples up to a natural isomorphism.
\end{itemize}
\end{theorem}

\begin{rem}
\label{R:proof_explanations}
The self contained proof of this theorem in such a form is given in \cite[Ch. 9, 10]{Zheglov_book}, and it uses two other correspondences between these data and equivalence classes of Schur pairs (we recall the definition below, as they will play a key role in our statements). The proof given in \cite{Zheglov_book} is an elaborated version of Mulase's proof from \cite{Mu} in a spirit of works \cite{Parshin2001}, \cite{Os}  and their  higher dimensional generalisations in \cite{Zheglov2013}, \cite{KOZ2014}.  In one direction, from the subring $B$ to geometric spectral data, the construction of this correspondence is the following: consider the Rees ring $\tilde{B}=\oplus_{i\ge 0}B_{i}s^i\subset B[s]$, where $B_i:=\{P\in B |\quad \ord (P)\le ir\}$. The curve $C$ is defined as $C:=\Projj \tilde{B}$. The ideal $(s)$ defines a smooth point $p$ on $C$. Analogously, the spectral module $F$  is endowed with a natural filtration given by the order function: $F_i:=\{ f\in F| \quad \ord (f)\le i\}$. With this filtration we can associate a  graded module and corresponding associated spectral sheaf:  
$$
 {}^{(-1)}\tilde{F}:=\oplus_{i\ge 0}F_{ir-1}s^i, \quad \cf :=\Projj {}^{(-1)}\tilde{F}.
$$ 
In other direction, from the spectral data to the ring $B$, one can use the Krichever map from remark \ref{R:Krichever_map} below to construct a corresponding Schur pair (see definition \ref{D:Schur_pair} below), and then theorem \ref{T:classif_Schur_pairs} to construct the ring $B$.

Recall that this classification has a long history: the first classification results of commuting pairs of ODOs with coprime orders appeared in works \cite{BC1}-\cite{BC3}. After that Krichever in works \cite{Kr1}, \cite{Kr2} gave classification of commutative subrings of ODOs  of any rank in general position in terms of geometric data. His version of classification theorems had a more analytical nature. The other versions (have more algebraic nature) are due to Mumford \cite{Mumford_article}, Drinfeld \cite{Dr}, Verdier \cite{Verdier} and Mulase \cite{Mu} (cf. also an important paper by Segal and Wilson \cite{SW}). 
\end{rem}

Let $B\subset D_1$ be an elliptic subring. Then by Schur theory from \cite{Schur}, cf. \cite[T. 4.6, C. 4.7]{Zheglov_book}, there exists an invertible operator $S=s_0+s_1\partial^{-1} +\ldots $ in the usual (Schur's) ring of pseudo-differential operators $E=K[[x]]((\partial^{-1}))$ such that $A:=S^{-1}BS \subset K((\partial^{-1}))$. Consider the homomorphism (of vector spaces)
\begin{equation}
\label{E:Sato_hom}
E\rightarrow E/xE \simeq K((\partial^{-1}))
\end{equation}
(sometimes it is called {\it the Sato homomorphism}). It defines a structure of an $E$-module on the space  $K((\partial^{-1}))$: for any $P\in K((\partial^{-1}))$, $Q\in E$ we put  $P\cdot Q= PQ$ (mod $xE$). 

Analogously, the homomorphism 
$$
1\circ :\tilde{K}[A_1]((\tilde{D}^{-1}))\simeq \tilde{K}^{\oplus q}((\tilde{D}^{-1})) \subset E_q\rightarrow  \tilde{K}((\tilde{D}^{-1})), \quad \sum_l p_l\tilde{D}^l \mapsto  \sum_l p_{l,0}\tilde{D}^l
$$
(cf. lemma \ref{L:conjugation in E_q}) defines a structure of a $\tilde{K}[A_1]((\tilde{D}^{-1}))$-module on the space  $\tilde{K}((\tilde{D}^{-1}))$: for any $P\in \tilde{K}((\tilde{D}^{-1}))$ and $Q\in \tilde{K}[A_1]((\tilde{D}^{-1}))$ we put $P\cdot Q= 1\circ (PQ)$. Analogously, $K((\tilde{D}^{-1}))\subset \tilde{K}((\tilde{D}^{-1}))$ is a right $K^{\oplus q}((\tilde{D}^{-1}))$-module. 

Now define the space $W:= F\cdot S \subset K((\partial^{-1}))$ (here $F$ is the same as in theorem \ref{T:A.Z 2.1}). Note that $W$ is an $A$-module, where the module structure is defined via the multiplication in the {\it field} $K((\partial^{-1}))$ and this module structure is induced by the $E$-module structure on $K((\partial^{-1}))$, because $K((\partial^{-1}))\subset E$ and $W\cdot A=(F\cdot S)\cdot (S^{-1}BS)=F\cdot (BS)=(F\cdot B)\cdot S= F\cdot S= W$. 
Note also that the modules $W$ and $F$ are isomorphic ($W$ is an $A$-module, $F$ as a $B$-module, and clearly $A\simeq B$). For convenience of notation, we will replace $\partial^{-1}$ by $z$ in the field $K((\partial^{-1}))$, i.e. $A,W\subset K((z))\simeq K((\partial^{-1}))$. 

Analogously, we can define the space $W':=F'\cdot S\subset \tilde{K}((\tilde{D}^{-1}))$, where $F'=\tilde{K}[\tilde{D}]$ and $S\in \tilde{K}[A_1]((\tilde{D}^{-1}))$ is an operator from lemma 
\ref{L:conjugation in E_q} or $W'':=F''\cdot S\subset K((\tilde{D}^{-1}))$, where $F''=K[\tilde{D}]$ if $S\in K^{\oplus q}((\tilde{D}^{-1}))$. 

If $B'=\hat{\Phi}(S^{-1}BS)$, where $S$ is Schur operator from proposition \ref{T:A.Z 7.2} (constructed for some monic differential operator from $B$), and $A':=S^{-1}B'S$, where $S$ is an operator from lemma \ref{L:conjugation in E_q} (constructed for some operator operator $P'\in B'$), and, by remark \ref{R:S_over_K} $S\in K^{\oplus q}((\tilde{D}^{-1}))$, so that $A'\subset K((\tilde{D}^{-1}))$, 
then, clearly, the modules $F''$ and $F$ are isomorphic ($F''$ as a $B'$-module, $F$ as a $B$-module), and  
$W''$ and $F$ are isomorphic ($W''$ is a $A'$-module, $F$ as a $B$-module), so also $W''\simeq W$. Moreover, all these modules are isomorphic as filtered modules (with respect to the order filtration). 

For subrings in $K((z))$ we can introduce the same notion of rank as for subrings in $D_1$\footnote{here $z$ will play a role of $\partial^{-1}$ or $\tilde{D}^{-1}$}:

\begin{Def}
\label{D:Q5.1} 
Let $A$ be a $K$-subalgebra of $K((z))$, and $r\in \dn$. $A$ is said to be an algebra of rank $r$ if 
$r= \gcd (\ord(a)|\quad  a\in A)$, where the order is defined in the same way as the usual order  in $D_1$ (cf. remark \ref{R:ord=deg}).
\end{Def}

\begin{Def}
\label{D:support} 
Let $W$ be a $K$-subspace in $K((z))$. The {\it support} of an element $w\in W$ is its highest symbol, i.e. $\sup (w):=HT(w)z^{-\ord (w)}$. The {\it support} of the space is\\
 $\Sup W:=\langle \sup (w)|\quad w\in W\rangle$. 
\end{Def}

\begin{Def}
\label{D:Schur_pair}
An embedded Schur pair of rank $r$  is a pair $(A,W)$ consisting of 
\begin{itemize}
\item
$A\subset K((z))$ a $K$-subalgebra of rank $r$ satisfying $A\cap K[[z]]=K$;
\item 
$W\subset K((z))$ a $K$-subspace with $\Sup W= K[z^{-1}]$
\end{itemize}
such that $W\cdot A\subseteq W$.
\end{Def} 

So, to any elliptic subring $B\subset D_1$ we can associate an embedded Schur pair. 

\begin{Def}
\label{D:equiv_Schue_pair}
Two embedded Schur pairs   $(A_i, W_i)$, $i=1,2$ of rank $r$ are {\it equivalent} if there exists an admissible operator $T$ such that $A_1=T^{-1}A_2T$, $W_1=W_2\cdot T$. An operator $T=t_0+t_1\partial^{-1}+\ldots $ is called {\it admissible} if $T^{-1}\partial T\in K((\partial^{-1}))$. 
\end{Def}

\begin{theorem}
\label{T:classif_Schur_pairs}
There are one-to-one correspondences  $[B]\longleftrightarrow [(A,W)]$ and $[B]/\thicksim\longleftrightarrow [(A,W)]/\thicksim$, where $\thicksim$ means the equivalence from theorem \ref{T:classif2} for rank one data.
\end{theorem}

A self contained proof of this theorem see e.g. in \cite[10.3]{Zheglov_book}.  In one direction, from $B$ to the Schur pair, the construction of this correspondence is described above. In other direction the Sato theorem is used (see \cite{SN}, cf. \cite[Appendix]{Mu}) to find the corresponding Sato operator $S$ such that $W=F\cdot S$ and then $B=SAS^{-1}$.

Let $C$ be a projective curve over $K$ and $p\in C$ be a regular $K$-point. Then $C_0:=C\backslash p$ is an affine curve. According to theorem \ref{T:classif2} for any torsion free rank one sheaf $\cf$ on $C$ there exists a normalised commutative elliptic ring of differential operators $B\subset D_1$ defined uniquely up to a scale transform, such that $B$ is isomorphic to the ring of regular functions on $C_0$, $B\simeq \co_{C_0}(C_0)$. Vice versa, any such ring $B$ is isomorphic to the ring of regular functions on some affine curve $C_0$, which can be compactified with the help of one regular $K$-point. 

\begin{Def}
\label{D:affine_spec_curve}
We'll call an affine curve $C_0$ over $K$ as {\it affine spectral curve} if it can be compactified with the help of one regular $K$-point, i.e. if there exists a projective curve $C$ over $K$ and a regular $K$-point $p$ such that $C_0\simeq C\backslash p$ (note that such $C$ is uniquely defined up to an isomorphism, see e.g. \cite[Ch1., \S 6]{Ha}).
\end{Def}

\begin{rem}
\label{R:Krichever_map}
There is the Krichever map 
$$
\chi_0:(C,p,\cf ,\pi , \hat{\phi} ) \rightarrow (A,W)
$$
defined for any coherent torsion free sheaf $\cf$ and any trivialisations $\pi, \hat{\phi}$ (for details see \cite[Ch.10]{Zheglov_book}). If $\cf =\co_C$, then $W=A$ and the rank of $A$ is $1$. Moreover, elements of $A$ are Laurent series expansions of functions from $\co_C(C\backslash p)$, s.t. the orders of elements of $A$ are the pole orders of corresponding functions. Differential operators corresponding to elements of $A$ via the correspondence from theorems \ref{T:classif2}, \ref{T:classif_Schur_pairs} have the same order as the elements. 

If $\rk (\cf )=1$, then $\rk (A)=1$ and the data $(C,p,\cf ,\pi )$ uniquely determines the class $[B]$ from theorem \ref{T:classif2}, see \cite[Rem. 10.24]{Zheglov_book}. 

If $\rk (A)=1$, then we can always choose a system of generators $a_1, \ldots , a_m$ of $A$ as a $K$-algebra,  such that $\ord (a_1)$ is coprime with the orders of other generators. Without loss of generality $a_i$ can be assumed to be monic. Conjugating $A$ by a suitable admissible operator (and using the usual Schur theory), we can get $a_1=z^{-\ord (a_1)}$. 

Since any pair of generators $(a_1, a_i)$ correspond to some differential operators of coprime orders, they are algebraically dependent and satisfy  some equations of Burchnall-Chaundy type $f_i(X,Y) =  X^q\pm Y^p+\ldots =0$ (see remark \ref{R:algebraic_dependence_of_normal_forms}). Vice versa, it's easy to see, using standard arguments from Hensel's lemma, that this equation uniquely determines a monic element $a_i$ from $A$ of a given order such that $f_i(a_i, z^{-\ord (a_1)})=0$. Thus, the equations $f_i$ completely determine the subring $A$ in $K((z))$. 

All correspondences in the classification theorems \ref{T:classif2}, \ref{T:classif_Schur_pairs} are relevant to the extension of scalars, i.e. if the sheaf $\cf$ with vanishing cohomologies is defined over an extension $\bar{K}$ of $K$, then the corresponding Schur pairs and the ring of commuting differential operators are determined over $\bar{K}$, and vice versa. As we'll see below, the corresponding  normalised normal forms will be determined over $\bar{K}$ as well.
\end{rem}

Let $B \subset D_1$ be an elliptic commutative subring of ODOs. Let $P_1, \ldots ,P_m$ be its monic generators over $K$ (any generators have constant highest coefficients, cf. \cite[Ch.3]{Zheglov_book}), such that $\Ord (P_1)=q$ is coprime with the order of $P_2$.
$$
B\simeq K[P_1,\ldots ,P_m]\simeq K[T_1,\ldots ,T_m]/I,
$$
where $I=(f_1,\ldots ,f_k)$ is a prime ideal, $f_i\in K[T_1,\ldots ,T_m]$. By  proposition \ref{P:normalized_Schur} and lemma \ref{L:normalisation} there exists a uniquely determined Schur operator $S\in \hat{D}_1^{sym}$ such that $B':=S^{-1}BS\in C(\partial^q)$ and $P_2'=S^{-1}P_2S$ is a normalised normal form of $P_2$ with respect to $P_1$. From lemma \ref{L:normalisation} we immediately get that the coefficients of all other normal forms $P_i'$ are uniquely determined. Obviously, $f_i(P_1', \ldots , P_m')=0$ for all $i$, and therefore define a set of equations on the coefficients of $\hat{\Phi}(P_i')\in \mathfrak{B}$ in an affine space. Note that any point of the affine algebraic set determined by these equations together with equations coming from commutation relations between $P_2', \ldots , P_m'$ 
defines a set of coefficients of commuting operators from $C(\partial^q)$, and these operators generate a ring isomorphic to $B$\footnote{with the isomorphism sending them to $T_1, \ldots ,T_m$}. 

\begin{Def}
\label{D:normalised_normal_form}
We'll call a commutative ring $B'\in C(\partial^q)\subset \hat{D}_1^{sym}$ generated over $K$ by monic operators $P_1'=\partial^q, P_2', \ldots, P_m'$, where $P_2'$ is normalised and $q$ is coprime with the order of $P_2'$,  as a {\it normalised normal form} of rank one with respect to a (ordered) set of generators $P_1',\ldots , P_m'$, and we'll call the coefficients of the operators $P_i'$ as coordinates of $B'$ (cf. lemma \ref{L:normalisation}). 

We'll denote by $X_{[B']}$ the corresponding affine algebraic set determined by the relations on coefficients of operators $P_i'$ (it is defined by an isomorphism class of the ring $B'$). 
\end{Def}

\begin{lemma}
\label{L:solution_of_BC_equation}
Let $f(X,Y) =  X^q\pm Y^p+\ldots \in K[X,Y]$ be a Burchnall-Chaundy polynomial with coprime $p,q$. Assume $P'\in C(\partial^q)\subset \hat{D}_1^{sym}\hat{\otimes}\tilde{K}$ is a monic operator with $\Ord (P')=p$ such that $f(P',\partial^q)=0$. 

Then $a:= S^{-1}\hat{\Phi}(P')S\in \tilde{K}((\tilde{D}^{-1}))$, where $S$ is an operator from lemma \ref{L:conjugation in E_q}, and $a$ is the uniquely defined monic element in $\tilde{K}((\tilde{D}^{-1}))$, satisfying the equation $f(a,\tilde{D}^q)=0$. 
\end{lemma}

The proof is  obvious in view of remark \ref{R:Krichever_map} and corollary \ref{C:Schur_subring}. In particular, from theorem below it follows that any monic operator $P'$ from lemma together with $\partial^q$ form a normal from of some pair of differential operators. 

\begin{theorem}
\label{T:parametrisation}
Let $C_0$ be an affine spectral curve over $K$ and $C$ its one-point compactification. Assume 
$$
\co_C(C_0)\simeq K[w_1,\ldots ,w_m]\simeq K[T_1,\ldots ,T_m]/I,
$$
where $I=(f_1,\ldots ,f_k)$ is a prime ideal and the order of $w_1$ is coprime with the order of $w_2$,  and the images of $w_i$ under the Krichever map (after some choice of $\pi$) are monic (cf. remark \ref{R:Krichever_map}). 

Then there exist normalised normal forms of rank one $B'\simeq \co_C(C_0)$ with respect to the ordered set of generators $P_1'=\partial^q, \ldots , P_m'$, where $\Ord (P_i')=\ord (w_i)$ for all $i\ge 1$, and there is a one to one correspondence between closed $K$-points of the affine algebraic set $X_{[\co_C(C_0)]}$ and isomorphism classes of torsion free rank one sheaves $\cf$ on $C$ with vanishing cohomologies $H^0(C,\cf )= H^1(C,\cf )=0$.
\end{theorem}

\begin{proof}
Let $\cf$ be a  torsion free rank one sheaf on $C$ with vanishing cohomologies. 
By theorem \ref{T:classif2} the triple $(C,p,\cf )$ corresponds to uniquely defined {\it normalised}\footnote{Each equivalence class $[B]$ contains a {\it normalised representative}, which determines the whole class uniquely. For example, we can choose a monic differential operator of minimal positive order in $B$, and normalise it (conjugating by an appropriate function $f\in \hat{R}$).} commutative elliptic subring $B\subset D_1$ up to a scale transform. Besides, the generators $w_i$ corresponds to formally elliptic differential operators $P_1, \ldots ,P_m$ of the same orders, and there exists a scale transform that makes them monic. By  proposition \ref{P:normalized_Schur} and lemma \ref{L:normalisation} there exists a uniquely determined Schur operator $S\in \hat{D}_1^{sym}$ such that $B':=S^{-1}BS\in C(\partial^q)$ and $P_2'=S^{-1}P_2S$ is a normalised normal form of $P_2$ with respect to $P_1$, i.e. $B'$ is a normalised normal form w.r.t. $P_1', \ldots ,P_m'$ which have the same orders as $w_i$ or $P_i$. Note that the scale transform is compatible with conjugation by $S$ and that the coefficients of $\hat{\Phi}(P_i')$ are invariant under any scale transform of $P_i'$ for all $i$. Thus, $\cf$ determines a closed point of $X_{[\co_C(C_0)]}$.

Vice versa, any closed point of $X_{[\co_C(C_0)]}$ determines a normalised normal form $B'$ w.r.t. some $P_1', \ldots ,P_m'$ which have the same orders as $w_i$. By corollary \ref{C:Schur_subring} and remark \ref{R:S_over_K} there exists an operator $S\in E_q$ such that $A':=S^{-1}\hat{\Phi}(B')S\in K((\tilde{D}^{-1}))$. Note that $W'':=F''\cdot S$ has support equal to $F''$, because $S$ is a monic invertible operator of order zero. So, $(W'', A')$ form an embedded Schur pair of rank one\footnote{Since  $\hat{\Phi}(C(\partial^q))$  consists of Laurent polynomials whose coefficients at negative powers of $\tilde{D}$ are not constants (see remark \ref{R:vector_form}), and since $S$ is monic and $A'\subset K((\tilde{D}^{-1}))$, it follows that $A'\cap K[[\tilde{D}^{-1}]]=K$.}. By theorems \ref{T:classif_Schur_pairs} and \ref{T:classif2} this Schur pair determines a normalised commutative subring $B\subset D_1$ of rank one and a torsion free sheaf $\cf\simeq \Projj \tilde{W''}$ of rank one with vanishing cohomologies. 

As it was noticed above, the modules $F$, $F''$ and $W''$ are isomorphic as filtered modules, and all correspondences are compatible with the scale transform. By this reason the maps $p\in X_{[\co_C(C_0)]} \mapsto \cf$ and $\cf \mapsto p\in X_{[\co_C(C_0)]}$ are mutually inverse. 
\end{proof}

\begin{rem}
\label{R:moduli}
This result indicates that the moduli space of {\it spectral sheaves} of rank one, i.e. sheaves with vanishing cohomologies, is an affine open subscheme of the compactified Jacobian (cf. \cite{MM}, \cite{Souza}, \cite{Rego}). We hope to cover this issue, also in the higher rank case, in future works.
\end{rem}

\begin{ex}
\label{Ex:Wallenberg}
Let $L=\partial^2+u$, $P=4 \partial^3+6 u \partial+3 u^{\prime}$, where  $u(x)=\sum_{k=0}^{\infty} \frac{1}{k !} u_k \cdot x^k$, be  ODOs of orders 2 and 3. Then  $[L,P]=0$ iff $6u u' +u'''=0$ (see \cite{Wall}). It's easy to see that the coefficients $u_k$, $k\ge 3$ are uniquely determined by this equation for any choice of free parameters $u_0, u_1, u_2$.  The spectral curve 
of these operators is given by 
$$
P^2=16 L^3+4\left(-3 u_0^2-u_2\right) L-4 u_0^3+u_1^2-2 u_0 u_2,
$$
The normalised normal form of $P$ (written in G-form) with respect to $L$ is\footnote{This example was calculated by V.D. Busov in his master thesis at Lomonosov MSU}:
$$
P' =4 \partial^3 + 2u_0A_{1}\partial+u_1A_1+\frac{2u_0^2+u_2}{2}(-1+A_{1})\int .
$$
If we transfer $P'$ into the matrix form (see Lemma \ref{L:B is matrix algebra over C(B)}), we get
$$
\tilde{P}'=\psi\circ\hat{\Phi}(P')=\left(
\begin{array}{cc}
	u_1 & 2 \left(2 \tilde{D}^2+u_0\right) \\
	4 \tilde{D}^4-2 \tilde{D}^2 u_0-2 u_0^2-u_2 & -u_1 \\
\end{array}
\right)
$$
So, by fixing an equation of the spectral curve, we get one-dimensional affine algebraic set in $\da^3$ parametrising torsion free sheaves with vanishing cohomology groups.  
\end{ex}

\begin{ex}
	\label{Ex: Rank 1 normal form}
	On the other hand, if we have two operators $L_2,L_3\in D_1$ satisfying the equation $L_3^2=4L_2^3-g_2L_2-g_3$, where $g_2,g_3$ are two constant coefficients in $K$ (may assume $L_2$ is normalized), then apparently $[L_2,L_3]=0$ (cf. \cite{BC1}). 
	
	Consider the normal form $L_3'$ of $L_3$ with respect to $L_2$. We fix $k=2,\xi=-1$ in this example, according to Lemma \ref{L:homog_comm}, $L_3'$ should be in the form of (within $E_2$):
	$$
	\hat{\Phi}(L_3')=2\tilde{D}^3+(c_{1,0}+c_{1,1}A_1)\tilde{D}+(c_{0,0}+c_{0,1}A_1)+(c_{-1,0}+c_{-1,1}A_1)\tilde{D}^{-1}
	$$
	with $c_{-1,0}+c_{-1,1}=0$.
	
	The calculations of $L_3'$ can be  separated into the following parts:
	\begin{enumerate}
		\item Written $L_2$ into G-form with unknown coefficients $t_i$, notice it's totally free of $B_j$(because $L_2\in D_1$, and by theorem \ref{T: P in hat(D) is a dif_op}), then use lemma \ref{L:H tfree of Bj iff xH tfree of Bj}:
		$$
		\hat{\Phi}(L_2)=\tilde{D}^2+t_0+t_1\Gamma_1\tilde{D}^{-1} +(t_2\Gamma_1-t_2\Gamma_2)\tilde{D}^{-2}+\cdots
		$$
		and correspondingly $L_3$ with unknown coefficient $m_j$
		$$
		\hat{\Phi}(L_3)=2\tilde{D}^3+m_0\tilde{D}+m_1+m_2\Gamma_1+(m_3\Gamma_1+m_4\Gamma_2)\tilde{D}^{-1}+\cdots
		$$
		Substitute the expression of $(L_2,L_3)$ into equation $L_3^2=4L_2^3-g_2L_2-g_3$, we get the relation that\footnote{We can substitute $\hat{\Phi}(L_3)$ and $\tilde{D}^2$ instead, in  this case calculations are easy in both cases.} 
		$$
		\begin{cases}
			m_0=3 t_0;
			\\m_1=\frac{3 t_1}{2};m_2=3 t_1;
			\\m_3=0;m_4=\frac{3}{2} (g_2-3 t_0^2);
		\end{cases}
		$$
		with the extra relations with parameter $g_2,g_3$ that 
		$$
		\begin{cases}
			t_2=\frac{1}{2} (3 t_0^2-g_2)
			\\-2 g_2 t_0+4 g_3+2 t_0^3+t_1^2=0
		\end{cases}
		$$
		\item Calculate the Schur operator $S$ in the equation $L_2S=S\partial^2$.
		\item Calculate the normal form $L_3'$ in the equation $L_3S=SL_3'$ (these calculations were made with the help of Wolfram Mathematica package)
	\end{enumerate}
	we get 
	$$
	\hat{\Phi}(L_3')=2 \tilde{D}^3+ t_0 A_1\tilde{D} -\frac{1}{2} t_1 A_1-\frac{1}{4}(1-A_1) \left(g_2-t_0^2\right)\tilde{D}^{-1};
	$$
	
	Transferring $L_3'$ into matrix form, we have 
	$$
	\psi\circ \hat{\Phi}(L_3')=\left(
	\begin{array}{cc}
		-\frac{t_1}{2} & 2 \tilde{D}^2+t_0 \\
		\frac{1}{2} \left(4 \tilde{D}^4-2 \tilde{D}^2 t_0-g_2+t_0^2\right) & \frac{t_1}{2} \\
	\end{array}
	\right)
	$$
\end{ex}

\noindent J. Guo,  School of Mathematical Sciences, Peking University and Sino-Russian Mathematics Center,  Beijing, China 
\\ 
\noindent\ e-mail:
$123281697@qq.com$

\vspace{0.5cm}

\noindent A. Zheglov,  Lomonosov Moscow State  University, Faculty
of Mechanics and Mathematics, Department of differential geometry
and applications, Leninskie gory, GSP, Moscow, \nopagebreak 119899,
Russia
\\ \noindent e-mail
 $azheglov@math.msu.su$, $alexander.zheglov@math.msu.ru$, $abzv24@mail.ru$

\end{document}